\documentclass[11pt]{article}

\usepackage{url,graphicx,tabularx,array,geometry}
\usepackage{graphicx}
\usepackage{amssymb}
\usepackage{fancyhdr}
\usepackage{wrapfig}
\usepackage{epstopdf}
\usepackage{amsmath}
\usepackage[labelfont=bf]{caption}
\usepackage[hidelinks]{hyperref}
\usepackage{amsthm}
\usepackage{bbm}
\usepackage[T1]{fontenc}
\usepackage{yfonts}
\usepackage{harpoon}
\usepackage[all]{xy}
\usepackage{tikz}
\usepackage{tkz-graph}
\usepackage{float}
\usetikzlibrary{arrows}
\usetikzlibrary{calc}
\usetikzlibrary{decorations.markings}
\tikzstyle{vertex}=[circle, draw, inner sep=0pt, minimum size=6pt]
\newcommand{\vertex}{\node[vertex]}

\renewcommand{\title}[1]{\textbf{#1}}
\renewcommand{\line}{\begin{tabularx}{\textwidth}{X>{\raggedleft}X}\hline\\\end{tabularx}\\[-0.5cm]}

\newcommand{\Aut}[1]{\operatorname{Aut}(#1)}

\newcommand{\Fit}[1]{\operatorname{Fit}(#1)}
\newcommand{\Syl}[2]{\operatorname{Syl}_{#1}(#2)}
\newcommand{\II}{{\mathcal I}}
\newcommand{\DD}{{\mathcal D}}
\newcommand{\OO}{{\mathcal O}}
\newcommand{\CC}{{\mathcal C}}
\newcommand{\NN}{{\mathcal N}}

\makeatletter
\def\imod#1{\allowbreak\mkern10mu({\operator@font mod}\,\,#1)}
\makeatother
\newcommand{\superscript}[1]{\ensuremath{^{\textrm{#1}}}}

\theoremstyle{plain}
\newtheorem{theorem}{Theorem}
\newtheorem{lemma}[theorem]{Lemma}
\newtheorem{corollary}[theorem]{Corollary}
\newtheorem{theoremN}{Theorem}[section]
\newtheorem{propositionN}[theoremN]{Proposition}
\newtheorem{corollaryN}[theoremN]{Corollary}
\newtheorem{lemmaN}[theoremN]{Lemma}

\theoremstyle{definition}
\newtheorem*{definition}{Definition}
\theoremstyle{remark}

\newtheorem*{remark}{Remark}
\newtheorem*{note}{Note}

\newcommand{\tightoverset}[2]{%
  \mathop{#2}\limits^{\vbox to -.5ex{\kern-0.75ex\hbox{$#1$}\vss}}}
\newcommand{\tikzoverset}[2]{%
  \tikz[baseline=(X.base),inner sep=0pt,outer sep=0pt]{%
    \node[inner sep=0pt,outer sep=0pt] (X) {$#2$};
    \node[yshift=1pt] at (X.north) {$#1$};
}}
\renewcommand\vec[1]{\tikzoverset{\rightharpoonup}{#1}}
\newcommand\PrG[1]{\Gamma_{\!#1}}
\newcommand\PG{\PrG{G}}
\newcommand\FrG[1]{\vec{\Gamma_{\!#1}}}
\newcommand\FG{\FrG{G}}

\begin{document}
\date{ }
\setcounter{page}{1}

\vspace*{\fill}
\title{\Large \begin{center}{\huge A Characterization of the Prime Graphs\\ of Solvable Groups. }\end{center} }

\noindent\line

\begin{center}{\Large Alexander Gruber\superscript{a}, Thomas M. Keller\superscript{b}, Mark Lewis\superscript{c},\\ Keeley Naughton\superscript{d}, Benjamin Strasser\superscript{e}}\end{center}
\begin{center}
\begin{tabular}{lc}
\superscript{a} & Department of Mathematical Sciences, University of Cincinnati\\& 2815 Commons Way, Cincinnati, OH 45221-0025, USA \\
\superscript{b} & Department of Mathematics, Texas State University\\& 601 University Drive, San Marcos, TX 78666-4616, USA \\
\superscript{c} & Department of Mathematical Sciences, Kent State University\\& Kent, OH 44242, USA\\
\superscript{d} &  Mathematics Department, Syracuse University\\&  805 South Crouse Avenue, Syracuse, NY 13244-1150, USA \\
\superscript{e} & Department of Mathematics, Carleton College\\& 1 North College Street, Northfield, MN 55057-4001, USA
\end{tabular}
\end{center}
\vspace*{\fill}

\thispagestyle{empty}
\pagestyle{fancy}
\rhead{\footnotesize {\emph{Preprint. May 2013.}} }
\pagebreak
\setcounter{page}{1}

\begin{abstract}
Let $\pi(G)$ denote the set of prime divisors of the order of a finite group $G$.  The \emph{prime graph} of $G$, denoted $\PG$, is the graph with vertex set $\pi(G)$ with edges $\{p,q\}\in E(\PG)$ if and only if there exists an element of order $pq$ in $G$.  In this paper, we prove that a graph is isomorphic to the prime graph of a solvable group if and only if its complement is $3$-colorable and triangle free.  We then introduce the idea of a minimal prime graph.  We prove that there exists an infinite class of solvable groups whose prime graphs are minimal.  We prove the $3k$-conjecture on prime divisors in element orders for solvable groups with minimal prime graphs, and we show that solvable groups whose prime graphs are minimal have Fitting length $3$ or $4$.
\end{abstract}

\section{Introduction.}
				
Prime graphs originated in the 1970s as a by-product of certain cohomological questions posed by K.W. Gruenberg.  Shortly after their introduction, prime graphs became objects of interest in their own right, and since then numerous contributions have been made to the topic.  The prime graphs of finite simple groups are well understood (see \cite{Vasilev}, \cite{Kondratev}, and \cite{Williams}); as is the structure of groups with acyclic prime graphs (see \cite{LucidoTree}).  Graph invariants such as diameter \cite{LucidoDiameter} and degree sequence \cite{Moghaddamfar} have also been extensively documented.  The question of how graph theoretic properties influence group structure remains, for the most part, open, and it is from this angle that our investigation proceeds.

Solvable groups possess several properties that motivate an extended discussion of their prime graphs.  Philip Hall established in \cite{Hall} that $G$ is solvable if and only if $G$ contains a Hall $\pi$-subgroup for every subset $\pi\subset \pi(G)$.  Graph theoretically, we can interpret this as the statement that, whenever $G$ is solvable, every induced subgraph $\PG[\pi]$ is the prime graph of a Hall $\pi$-subgroup of $G$.

Of further use is the following proposition, which we refer to as \emph{Lucido's Three Primes Lemma}.

\begin{lemma}[Lucido's Three Primes Lemma, \cite{LucidoDiameter}]
Let $G$ be a finite solvable group.  If $p,q,r$ are distinct primes dividing $|G|$, then $G$ contains an element of order the product of two of these three primes.
\end{lemma}

\noindent Equivalently, if $G$ is solvable, then $\PG$ cannot contain an independent set of size $3$.  In other words, $\overline\PG$ must be triangle free.

Williams observed in \cite{Williams} that every solvable group with a disconnected prime graph must be either a Frobenius or $2$-Frobenius group.  It follows that whenever an edge $pq$ is missing from the prime graph of a solvable group, the corresponding Hall $\{p,q\}$-subgroup $H_{pq}$ admits a fixed point free action between either the Sylow subgroups of $H_{pq}$ or their image in its Fitting quotient.  Our characterization begins by defining an acyclic orientation of $\overline{\PG}$ indicating the direction of this action for every edge $pq \in \overline{\PG}$.  We refer to this as the \emph{Frobenius digraph} of $G$.

The Frobenius digraph affords us a powerful tool for understanding the prime graphs of solvable groups.  The primary result of this paper can be summarized by the following theorem, though in fact both implications are derived from stronger results.

\begin{theorem}
A graph $F$ is the prime graph of some solvable group if and only if its complement $\overline{F}$ is $3$-colorable and triangle free.
\end{theorem}

Our characterization has some interesting applications.  For example, we find that the girth of the prime graph of a solvable group is exactly equal to three with only a few exceptions, which we then classify.

\begin{corollary}\label{cor1}
The prime graph of any solvable group has girth $3$ aside from the following exceptions: the $4$-cycle, the $5$-cycle, and the $7$ unique forests that do not contain an independent set of size $3$.
\end{corollary}

In the remaining sections of the paper, we present an extended application of how properties of solvable groups may be derived from their prime graphs using the Frobenius digraph and our main theorem.  We introduce the class of prime graphs that are minimal with respect to the property that they are isomorphic to the prime graph of some solvable group.  We are able to show that any solvable group of order $n$ with a minimal prime graph contains an element whose order is divisible by at least one third of the primes dividing $n$.  As our final result, we show that any solvable group with a minimal prime graph has Fitting length of $3$ or $4$, after which we present examples of solvable groups having minimal prime graphs and both possible Fitting lengths.  This final result is reminiscent of Lucido's work on solvable groups with prime graphs of diameter $3$ in \cite{LucidoDiameter}, which also have Fitting length $3$ or $4$. If $G$ is a solvable group, we denote the Fitting length of $G$ by $\ell_F (G)$.

\begin{theorem}
Let $G$ be a solvable group with a minimal prime graph.  Then $3 \leq \ell_F(G) \leq 4$.
\end{theorem}

Before we begin, let us briefly introduce the notational conventions to be used throughout this paper.  Unless stated otherwise, all graphs will be assumed simple and all groups finite.  In an undirected graph $\Gamma$, we denote edges $\{p,q\}\in E(\Gamma)$ by $pq\in \Gamma$.  For edges in a directed graph $\vec{\Gamma}$, we write edges from $p$ to $q$ as either $pq\in \vec{\Gamma}$ or $p\rightarrow q$ when unclear.  Subgraphs of a graph $F$ induced by a subset $\pi\subseteq V(F)$ will be denoted $F [\pi]$.  When we refer to cycles or paths in a directed graph, it is implicitly assumed that these cycles and paths are directed.  We refer to paths on $n+1$ vertices with $n$ edges as $n$-paths.

In an undirected graph $\Gamma$, the \emph{$k$-neighborhood} of a vertex $v\in \Gamma$, which we will denote $N^k(v)$, is defined as the set of vertices $u\in \Gamma$ such that a path in $\Gamma$ exists connecting $u$ and $v$ and the shortest such path has length $k$.  In a directed graph $\vec{\Gamma}$, a distinction is made between $k$-in- and $k$-out-neighborhoods of $v$.  The $k$-in-neighborhood of $v$ is denoted $N_\uparrow^k(v)$ and consists of vertices $u\in \vec{\Gamma}$ for which there exists a directed path in $\Gamma$ beginning with $u$ and ending with $v$ and such that the shortest such path has length $k$.  The $k$-out-neighborhood of $v$, denoted $N_\downarrow^k(v)$ is defined similarly, but instead for paths beginning with $v$ and ending in $u$.

We will sometimes write subgroups in the Fitting series of $G$ by $F_k(G)$ (or $F_k$ when there is no ambiguity), so that $F_1 (G) = \Fit{G}$, $F_2/F_1 = \Fit {G/F_1}$, and so on.  For a prime $p$ dividing $|G|$, unless stated otherwise, we denote by $P$ an arbitrary Sylow $p$-subgroup of $G$, and all statements about $P$ apply to every Sylow $p$-subgroup of $G$.  Finally, unless stated otherwise, for a set $\pi$ of primes dividing $|G|$, we denote by $H_\pi$ a Hall $\pi$-subgroup of $G$, unless $\pi$ consists of only two (or three, resp.) primes $p$ and $q$ (and $r$), in which case we write $H_{pq}$ (or $H_{pqr}$).

\section{Characterization.}

In this section, we characterize the prime graphs of solvable groups.

We would like to determine as much as possible about the way that Sylow subgroups of a solvable group $G$ interact given information from $\PG$.
 We know from \cite[Thm.~A]{Williams} that any solvable group with a disconnected prime graph is Frobenius or $2$-Frobenius.  Therefore, we can pick out disconnected subgraphs of $\PG$ to find Hall subgroups that contain fixed point free action.  We begin by defining $2$-Frobenius groups, introducing some new terminology, and providing additional details regarding their structure.

\begin{definition}
A group $G$ is a \emph{$2$-Frobenius group} if $F_2$ and $G/F_1$ are Frobenius groups, where $F_1 = \Fit{G}$ and $F_2/F_1=\Fit{G/F_1}$.  We will often refer to the Frobenius kernel of $G/F_1$ as the \emph{upper kernel} of $G$ and the Frobenius kernel of $F_2$ as the \emph{lower kernel} of $G$.
\end{definition}

We immediately see that in a $2$-Frobenius group $G$, the primes dividing $[F_2:F_1]$ are disjoint from those dividing $|F_1|$ or $[G:F_2]$.  In fact, we will see they form a clique in $\PG$. 

\begin{lemmaN}\label{2 Frobenius}
Let $G$ be a $2$-Frobenius group where $F_1 = F(G)$ and $F_2/F_1 = F (G/F_1)$.  Then $G/F_2$ and $F_2/F_1$ are cyclic groups, $F_1$ is not a cyclic group, and the upper kernel of $G$ is a cyclic group of odd order.
\end{lemmaN}

\begin{proof}
$F_2/F_1$, which we will write $H$, is the upper kernel of $G$.  Thus $H$ is nilpotent by Thompson's theorem.  We also know that $H$ is isomorphic to a Frobenius complement of $F_2$, so all Sylow subgroups of $H$ are cyclic or generalized quaternion.  Every Sylow subgroup of $H$ must then be a Frobenius kernel.  It is easy to see that a cyclic $2$-subgroup and a generalized quaternion group cannot be a Frobenius kernel, so the Sylow $2$-subgroup of $H$ is trivial.  All other Sylow subgroups are cyclic, so since $H$ is nilpotent, we conclude that $H$ is cyclic of odd order.

We know that $C_{G/F_1}(H)\leqslant H$ so $G/F_2$ is isomorphic to a subgroup of the automorphism group of $H$.  Since $H$ is a cyclic group of odd order, its automorphism group is cyclic, so $G/F_2$ is cyclic.  Finally, we know that $C_G(F_1)\leqslant F_1$, so $G/F_1$ is isomorphic to a subgroup of the automorphism group of $F_1$.  If $F_1$ is cyclic, its automorphism group is abelian.  Since $G/F_1$ is a Frobenius group, we conclude that $F_1$ is not cyclic.
\end{proof}

We notice that the simplest disconnected subgraphs of $\PG$ occur whenever an edge $pq\notin \PG$ as $\PG[\{p,q\}]$.  This tells us that a Hall $\{p,q\}$-subgroup $H_{pq}$ is either Frobenius or $2$-Frobenius.  Thus it is convenient to distinguish the following type of $2$-Frobenius groups.

\begin{definition}
If $G$ is a $2$-Frobenius group for which there are primes $p$ and $q$ so that $G/F_2$ and $F_1$ are $p$-groups and $F_2/F_1$ is a $q$-group, we say that $G$ is a \emph{$2$-Frobenius group of type $(p,q,p)$}.
\end{definition}

\begin{lemmaN}\label{complement}
If $H$ is a $2$-Frobenius group of type $(p,q,p)$ for primes $p$ and $q$, then $F_2(H)$ has a complement in $H$.  Furthermore, the semidirect product of $\Fit{H}$ with this complement is a Sylow $p$-subgroup of $H$.
\end{lemmaN}

\begin{proof}
Let $R = \Fit{H}$ and $S = F_2(H)$.  In light of the previous remark, $R$ is the Sylow $p$-subgroup of $S$.  Let $Q\in \Syl{q}{H}$.  Since $Q$ is a Frobenius complement to $R$ in $S$, we have $N_S (Q) = Q$.  By the Frattini argument, $H = N_H (Q) S = N_H (Q) QR = N_H (Q)R$.  Let $P$ be a Sylow $p$-subgroup of $N_H (Q)$.  Thus, $N_H (Q) = PQ$ and $H = N_H (Q) R = PQR = PS$, and $P \cap S = (P \cap N_H (Q)) \cap S = P \cap (N_H (Q) \cap S) = P \cap Q = 1$.  Thus, $P$ is a complement to $S$ in $H$.  Observe that $PR$ is a $p$-group and $|H:PR| = |Q|$, so $PR$ is a Sylow $p$-subgroup of $H$.
\end{proof}
\begin{corollaryN}\label{type (p,q,p)}
Suppose $H$ is a $2$-Frobenius group of type $(p,q,p)$ for primes $p$ and $q$.  If $P$ is a Sylow $p$-subgroup of $H$, then $P$ is not a Frobenius complement.  If $Q$ is a Sylow $q$-subgroup of $H$, then $Q$ is cyclic.
\end{corollaryN}

\begin{proof}
Observe that $Q$ is isomorphic to the upper Frobenius kernel of $H$, so $Q$ is cyclic by Lemma \ref{2 Frobenius}.  By Lemma \ref{complement}, $P$ is a nontrivial semidirect product.  This implies that $P$ has more than one subgroup of order $p$.  It is well known that Frobenius complements have unique subgroups of prime order.  Therefore, $P$ cannot be a Frobenius complement.
\end{proof}

The preceding lemmas provide a description of Hall $\{p,q\}$-subgroups for all nonedges in $pq\notin\PG$, but what do these facts mean together?  We are motivated to study the formation of edges in $\overline{\PG}$ and watch for emergent properties in the group structure, suspecting the whole to be greater than the sum of its parts.  To this end, we assign directions to the edges in $\overline {\PG}$.

\begin{definition}
Define an orientation of $\overline\PG$ for a finite solvable group $G$ as follows.  For each edge $pq\in\overline{\PG}$, a Hall $\{p,q\}$-subgroup $H_{pq}$ is either a Frobenius or $2$-Frobenius group by \cite[Thm.~A]{Williams}.  If $H_{pq}$ is a Frobenius group with complement a Sylow $p$-subgroup and kernel a Sylow $q$-subgroup, we direct the edge $pq$ in $\FG$ so that $p\rightarrow q$.  If $H_{pq}$ is a $2$-Frobenius group of type $(p,q,p)$, we direct the edge $pq$ in $\FG$ by $p\rightarrow q$.  We call this orientation the \emph{Frobenius digraph} of $G$, denoted $\FG$.
\end{definition}

\begin{remark}
We choose to direct the edges associated with $2$-Frobenius groups in $\FG$ based on the "higher" Frobenius action so that the orientation is preserved when taking factor groups.  This way, we are guaranteed that if $p \rightarrow q$ in $\FG$, the Frobenius kernel of either $H_{pq}$ or $H_{pq}/F_1(H_{pq})$ will be a Sylow $q$-subgroup.  It is also possible to define $\FG$ so that edges corresponding to $2$-Frobenius groups are oriented based on "lower" Frobenius action, that is, to direct $q\rightarrow p$ in $\FG$ if $H_{pq}$ is a $2$-Frobenius group of type $(p,q,p)$.
\end{remark}

When $r\rightarrow q$ and $q\rightarrow p$ in $\FG$, we notice that $\PG[\{p,q,r\}]$ is disconnected, so a Hall $\{p,q,r\}$-subgroup must be Frobenius or $2$-Frobenius.  We next define $2$-Frobenius groups of type $(p,q,r)$, which we then show are closely related to such $2$-paths $r\rightarrow q \rightarrow p$ in $\FG$.

\begin{definition}
Suppose that there exist distinct primes $p$, $q$, and $r$ so that $G = PQR$, where $P$, $Q$, and $R$ are Sylow $p$-, $q$-, and $r$-subgroups respectively, $PQ$ is a Frobenius group with kernel $P$, and $QR$ is either a $2$-Frobenius group of type $(r,q,r)$ or a Frobenius group with Frobenius kernel $Q$.  Then we say that $G$ is a \emph{$2$-Frobenius group of type $(p,q,r)$}.
\end{definition}

Observe that if $G$ is a $2$-Frobenius group of type $(p,q,r)$, then $\FG$ has the form $r \rightarrow q \rightarrow p$.  We next show that the converse is true, that is, that subgraphs of $\FG$ of the form $r \rightarrow q \rightarrow p$ correspond to $2$-Frobenius Hall subgroups of type $(p,q,r)$.

\begin{lemmaN}\label{2path}
If $r \rightarrow q \rightarrow p$ is a $2$-path in $\FG$ for a solvable group $G$, then a Hall $\{p,q,r\}$-subgroup $H_{pqr}$ is $2$-Frobenius of type $(p,q,r)$.
\end{lemmaN}

\begin{proof}
Let $H = H_{pqr}$.  The prime graph of $H$ has two connected components $\{ p, r \}$ and $\{ q \}$, so $H$ is Frobenius or $2$-Frobenius by  \cite[Thm.~A]{Williams}.  Suppose first that $H$ is a Frobenius group with kernel $K$ and complement $C$.  One of the connected components is the set of primes dividing $|C|$ and the other is the set of primes dividing $|K|$.  Thus either $K$ or $C$ is a Sylow $q$-subgroup of $H$.  Suppose $C$ is a Sylow $q$-subgroup of $H$.  Since $K$ is nilpotent, the Sylow $r$-subgroup $R$ of $K$ is normal in $K$.  However, then $C$ normalizes $R$ and $RC$ is a Frobenius group with kernel $R$, contradicting that $r\rightarrow q$ in $\FG$.  On the other hand, if $K$ is the Sylow $q$-subgroup of $H$, we see that $HP$ is a Frobenius group with kernel $H$, contradicting $q \rightarrow p$ in $\FG$.  Thus we conclude that $H$ is not Frobenius.

We now know that $H$ is $2$-Frobenius, so it remains to be shown that $H$ is of type $(p,q,r)$.  Observe that either $F_2/F_1$ is the Sylow $q$-subgroup of $H/F_1$ or $F_1$ and $H/F_2$ are both $q$-groups.  Suppose $F_1$ and $H/F_2$ are $q$-groups.  Let $C$ be a complement to $F_2$ in $H$, and let $R$ be a Sylow $r$-subgroup of $H$.  Since $G/F_1$ is a Frobenius group, $RF_1$ is normal in $G$.  We see that $CRF_1$ is a Hall $\{ q,r \}$-subgroup of $H$ and is a $2$-Frobenius group of type $(q,r,q)$, and we have $q \rightarrow r$ in $\FG$, a contradiction.  Thus, $F_2/F_1$ is the Sylow $q$-subgroup of $G$, which is cyclic by Lemma \ref{2 Frobenius}.  We note that a Hall $\{q, r\}$-subgroup of $H$ is either Frobenius or $2$-Frobenius of type $(r,q,r)$.  Since $q \rightarrow p$ in $\FG$, we know that a Hall $\{p, q\}$-subgroup of $H$ is either Frobenius or $2$-Frobenius of type $(p,q,p)$.  In latter case, we know that the Sylow $q$-subgroup is not cyclic by Corollary \ref{type (p,q,p)} and this contradicts the fact we have seen that a Sylow $q$-subgroup is cyclic.  Therefore, a Hall $\{p, q\}$-subgroup is a Frobenius group, and we conclude that $H$ has type $(p,q,r)$.
\end{proof}

If we can describe the Sylow subgroups of $2$-Frobenius groups of type $(p,q,r)$, we can read off the structure of Sylow $p$, $q$, and $r$- subgroups of $G$ whenever $r\rightarrow q\rightarrow p$ in $\FG$.

\begin{lemmaN}\label{type (p,q,r)}
Suppose that $p,q,r$ are distinct primes and $G=PQR$ is a $2$-Frobenius group of type $(p,q,r)$, where $P$, $Q$, and $R$ are Sylow $p$-, $q$-, and $r$-subgroups respectively. Then $P$ is not cyclic, $Q$ is cyclic, and $R$ is not generalized quaternion.
\end{lemmaN}

\begin{proof}
We observe that $Q$ is isomorphic to the upper kernel of $G$.  Thus $Q$ is cyclic by Lemma \ref{2 Frobenius}.  If $r$ does not divide $|F_1(G)|$, then $R$ is isomorphic to a Frobenius complement of $G/F_1$, and therefore cyclic by Lemma \ref{2 Frobenius}.  If $r$ does divide $|F_1 (G)|$, then $H_{qr}$ is $2$-Frobenius, so $R$ is not a Frobenius complement by Lemma \ref{type (p,q,p)}.  In both cases, we see that $R$ is not generalized quaternion.  By Lemma \ref{complement}, we know that $F_2 (G)$ has a complement $C$.  Observe that $C$ will normalize $PQ$.  It is not difficult to see that $PQC$ is a $2$-Frobenius group, and so, by Lemma \ref{2 Frobenius}, we see that $P$ is not cyclic.
\end{proof}

In fact, we can extend this result to gain even more information from $2$-paths in $\FG$.

\begin{corollaryN}\label{adjacent}
Let $G$ be a solvable group.  If $p_1 \rightarrow p_2 \rightarrow p_3$ is a $2$-path in $\FG$, then for every prime $q \in N_\uparrow^1 (p_3)$, a Hall $\{q, r\}$-subgroup $H_{qr}$ is a Frobenius group for every prime $r \in N_\downarrow^1 (q)$.
\end{corollaryN}

\begin{proof}
Let $q \in N_\uparrow^1(p_3)$ be any prime for which $N_\uparrow^1(q)$ is nonempty.  Note that $N_\uparrow^1 (p_2)$ is nonempty, so such a prime $q$ exists.  Consider a Hall $\{p_1,p_2,p_3\}$-subgroup $H_{p_1p_2p_3}$.  We have by Lemma \ref{2path} that $H_{sqp_3}$ is a $2$-Frobenius group of type $(p_3, p_2, p_1)$.  We know that the Hall $\{p_2, p_3\}$-subgroup $H_{p_2p_3}$  is a Frobenius group, and by Lemma \ref{type (p,q,r)}, $P_3$ is not cyclic.

With this in mind, let $q \in N_\uparrow^1(p_3)$ be arbitrary and consider a prime $r \in N_\downarrow^1 (q)$.  Let $H_{qr}$ be a Hall $\{q, r\}$-subgroup of $G$.  Note that the prime graph of $H_{qr}$ is disconnected, so $H_{qr}$ is either Frobenius or $2$-Frobenius of type $(r,q,r)$.  We show that it is Frobenius.  We first show the result if $r = p_3$.  If $H_{qp_3}$ is a $2$-Frobenius group of type $(q,p_3,q)$, then $P_3$ must be cyclic by Corollary \ref{type (p,q,p)}, contradicting Lemma \ref{type (p,q,r)}.  Thus $H_{qp_3}$ is a Frobenius group with complement $Q$.  For each remaining prime $r \in N_\downarrow^1(q)$ other than $p_3$, suppose that $H_{qr}$ is a $2$-Frobenius group of type $(q,r,q)$.  Then $Q$ cannot be a Frobenius complement, and we obtain a contradiction with  Corollary \ref{type (p,q,p)}. Hence $H_{qr}$ is Frobenius as well.
\end{proof}

Our investigation of paths in $\FG$ concludes with the following theorem, which strongly informs us of the types of structures that may occur in $\FG$.  This theorem constitutes the primary argument of one direction of the classification.

\begin{corollaryN}\label{3Pcor}
The Frobenius digraph of a solvable group cannot contain a directed $3$-path.
\end{corollaryN}

\begin{proof}
Suppose that $p_1 \rightarrow p_2 \rightarrow p_3 \rightarrow p_4$ is a $3$-path in the Frobenius digraph of $G$.  We have that $H_{p_1p_2p_3}$ is $2$-Frobenius of type $(p_1,p_2,p_3)$ by Lemma \ref{2path}.  If $P_2$ is a Sylow $p_2$-subgroup of $G$, then $P_2$ is cyclic by Lemma \ref{type (p,q,r)}.  On the other hand, $H = H_{p_2,p_3,p_4}$ is $2$-Frobenius of type $(p_2,p_3,p_4)$ by Lemma \ref{2path}.  Applying Lemma \ref{type (p,q,r)}, it follows that $P_2$ is not cyclic. This is a contradiction.
\end{proof}

\begin{remark}
It is obvious that the Frobenius digraph is acyclic when we assume that each arrow corresponds to a Frobenius group, since the order of a Frobenius complement divides the order of its kernel minus one\cite[Lem.~16.6]{HuppertCharacters}.  That this remains true when we allow the possibility of $2$-Frobenius groups is not as easy, but follows immediately from Corollary \ref{3Pcor}.
\end{remark}

Conversely, we show that any digraph violating neither Corollary \ref{3Pcor} nor Lucido's Three Primes Lemma is isomorphic to the Frobenius digraph of some solvable group.

\begin{theoremN}\label{PartialConverse}
For any $3$-colorable, triangle free graph $F$, there exists a solvable group $G$ for which $F$ is isomorphic to the complement of prime graph of $G$.  Furthermore, there exists an acyclic orientation of $F$ that does not contain a directed $3$-path, and given any such orientation $\vec{F}$, there exists a solvable group $G$ for which $\vec{F}$ is isomorphic to the Frobenius digraph of $G$.
\end{theoremN}

\begin{proof}
We begin by showing that if $F$ is a $3$-colorable, triangle free, then an acyclic orientation of $F$ exists that does not contain a $3$-path.  Take any $3$-coloring of $F$ and arbitrarily label the vertices with numbers $1$, $2$, and $3$ so that vertices of the same color have the same label.  Direct the edges of $F$ from lower to higher numbered colors. By construction, the resulting orientation is acyclic and contains no $3$-paths.

Now, let $\vec{F}$ be any acyclic orientation of $F$ that does not contain a directed $3$-path.  We now show that there is a solvable group $G$ whose Frobenius digraph is isomorphic to $\vec F$.  Let $\mathcal{O}$ be the set of vertices in $\vec{F}$ with in-degree $0$ and non-zero out-degree, $\mathcal{D}$ the set of vertices with both in- and out-degrees non-zero, and $\mathcal{I}$ the vertices with out-degree $0$.  (Here $\mathcal{O}$ reminds us of "outgoing" vertices, $\mathcal{D}$ reminds us of "double Frobenius" as by Lemma \ref{2path} vertices with this property imply the existence of a $2$-Frobenius Hall subgroup, and $I$ reminds us of "ingoing" vertices, including singleton vertices in $\mathcal{I}$.)  Denote the number of vertices in each of these sets by $n_o$, $n_d$, and $n_i$, respectively.

Let $\mathcal {P} = \{ p_j \in \mathbb{P} : j = 1, \ldots, n_o \}$ be a set of distinct primes and define $p = p_1p_2\cdots p_{n_o}$.  By Dirichlet's theorem, we can pick a set $\mathcal{Q} = \{ q_k \in \mathbb{P} : k = 1,\ldots,n_d \}$ of distinct primes such that $\displaystyle q \equiv 1 \imod{p}$ for every prime $q \in \mathcal{Q}$.  Define a directed graph $\vec{\Lambda}$ with vertex set $\mathcal{P} \bigcup \mathcal{Q}$ and edge set defined by the image of some fixed injective graph homomorphism $\Phi : \vec{F} [\mathcal{O}\bigcup \mathcal{D}] \rightarrow \vec{\Lambda}$ mapping vertices in $\mathcal{O}$ to primes in $\mathcal{P}$ and vertices in $\mathcal{D}$ to primes in $\mathcal{Q}$.  Let $T = C_{p_1} \times \ldots \times C_{p_{n_o}}$ and $U = C_{q_1} \times \ldots \times C_{q_{n_d}}$.  Since $p_j \mid q_k - 1$ for each pair of primes $p_j \in \mathcal{P}, q_k \in \mathcal{Q}$, we can define a semidirect product $K = U\rtimes T$ by allowing $C_{p_j}$ to act fixed point freely on $C_{q_k}$ if $p_jq_k\in \vec{\Lambda}$ and trivially otherwise.  It follows that each Hall $\{p_j,q_k\}$-subgroup of $K$ is a Frobenius group if $p_jq_k\in\vec{\Lambda}$ and a direct product otherwise.  Note that $\vec{\Lambda}$ is the Frobenius digraph of $K$.

For each vertex $v \in \mathcal{I}$, let $\Phi_1 (v)$ denote the set of primes in the image of $\vec{F} [N_\uparrow^1(v)]$ under $\Phi$, with $\Phi_2(v)$ defined analogously.  Note that $N_{\uparrow}^1(v) \bigcap N_{\uparrow}^2(v) = \emptyset$ by the assumption that $F$ is triangle free.  In the case that $\Phi_1(v) \not= \emptyset$, let $H_{v}$ be a Hall $\Phi_1 (v) \bigcup \Phi_2 (v)$-subgroup of $K$.  Then $\Fit{H_v}$ is a cyclic Hall $\Phi_1(v)$-subgroup of $K$.  Let $\mathcal {R} = \{ r_j \in \mathbb{P} : j = 1, \ldots, n_i \}$ be a set of distinct primes so that each $v_j \in \mathcal{I}$ is associated with a unique $r_j$.  When $\Phi_1(v_j) = \emptyset$, define $R_{v_j} = C_{r_j}$.  For the remaining vertices in $\II$, again by Dirichlet, we may insist that $r_j \equiv 1 \imod {|\Fit{H_{v_j}}|}$.  Then, by \cite[Lemma 1.8]{Keller1}, there exists a faithful irreducible $\mathbb{F}_{r_j} H_{v_j}$-module $R_{v_j}$ such that $\Fit {H_{v_j}}$ acts fixed point freely on $R_{v_j}$.  Finally, define a direct product $J = R_{v_1} \times \cdots \times R_{v_{n_i}}$.  Let $G = J \rtimes K$ be the semidirect product where any subgroup $C_s \leqslant K$ for $s \in \mathcal{P} \bigcup \mathcal{Q}$ acts on $R_{v_j}$ by the appropriate module action when $s \in \Phi_1(v_j)\bigcup \Phi_2(v_j)$ and trivially otherwise.  It follows that $\vec{F}$ is isomorphic to the Frobenius digraph of $G$.
\end{proof}

Note that any group constructed in the method described above from graph with chromatic number $3$ has Fitting length $3$.  We will return to this observation later in the paper during further examination of the connection between Fitting lengths and prime graphs.

As an immediate corollary to Theorem \ref{PartialConverse}, we observe that the prime graphs of most solvable groups contain a $3$-cycle, and in fact classify all exceptions.

\begin{corollaryN}\label{Girf}
The prime graph of any solvable group has girth $3$ aside from the following exceptions: the $4$-cycle, the $5$-cycle, and the $7$ unique forests that do not contain an independent set of size $3$.
\end{corollaryN}
\begingroup
\setlength\intextsep{0pt}
\begin{wrapfigure}[10]{r}{0.5\textwidth}
\begin{center}
	\begin{tikzpicture}[scale=0.7]
		\vertex (a) at (-2.5,2.5) {};

		\vertex (b) at (-0.707107, 1.79289) {};
		\vertex (c) at (-0.707107, 3.20711) {};
		\vertex (d) at (0.707107, 1.79289) {};
		\vertex (e) at (0.707107, 3.20711) {};

		\vertex (f) at (1.79289, 1.79289) {};
		\vertex (g) at (1.79289, 3.20711) {};
		\vertex (h) at (3.20711, 1.79289) {};
		\vertex (i) at (3.20711, 3.20711) {};
		\path
			(b) edge (d)
			(c) edge (e)
			(f) edge (h)
			(g) edge (f)
			(g) edge (i);
		
		\vertex (j) at (-3.20711,0) {};
		\vertex (k) at (-1.79289,0) {};		
		
		\vertex (l) at (4.29289, -0.418249) {};
		\vertex (m) at (5.707107, -0.418249) {};
		\vertex (n) at (5., 0.806497) {};

		\vertex (o) at (1.79289, -0.418249) {};
		\vertex (p) at (3.20711, -0.418249) {};
		\vertex (q) at (2.5, 0.806497) {};
		\path
			(l) edge (n)
			(o) edge (p)
			(n) edge (m);
	
		\vertex (r) at (-0.70711,0) {};
		\vertex (s) at (0.70711,0) {};		

		\vertex (t) at (4.29289, 3.20711) {};
		\vertex (u) at (4.29289, 1.79289) {};
		\vertex (v) at (5.70711, 1.79289) {};
		\vertex (w) at (5.70711, 3.20711) {};
			
		\vertex (x) at (7.5, 2.15) {};
		\vertex (y) at (6.64405, 1.52812) {};
		\vertex (z) at (6.97099, 0.521885) {};
		\vertex (aa) at (8.02901, 0.521885) {};
		\vertex (bb) at (8.35595,1.52812) {};
		\path
			(r) edge (s)
			(t) edge (w)
			(t) edge (u)
			(u) edge (v)
			(v) edge (w)
			(x) edge (y)
			(y) edge (z)
			(z) edge (aa)
			(aa) edge (bb)
			(bb) edge (x);
		\end{tikzpicture}
\end{center}
\caption{Exceptions to Corollary \ref{Girf}.}
\end{wrapfigure}
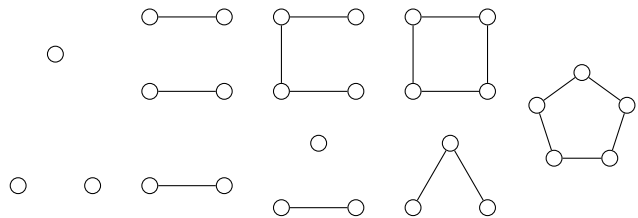{
\noindent\emph{Proof.} 
It is easily verifiable that each of the exceptional cases listed above have triangle free and $3$-colorable complements, so by Theorem \ref{PartialConverse} each can be realized as the prime graph of some solvable group.  Thus it remains to be shown that these exceptions are the only such prime graphs with girth not equal to $3$. 

Suppose that $\Gamma$ is a connected non-cycle triangle-free graph on $n\geq 5$ vertices with independence number $\alpha$, chromatic number $\chi$, and maximum vertex degree $\Delta$.  By Brooks' Theorem \cite{Brooks}, we have $\chi\leq \Delta$, whence \[\alpha \geq \frac{n}{\chi}\geq \frac{n}{\Delta} > \sqrt{n}.\]
Thus $2 < \sqrt{n} < \alpha$.  It follows from Lucido's Three Primes Lemma that $\Gamma$ cannot be the prime graph of a solvable group.  Likewise, any disconnected triangle free graph on $5$ or more vertices necessarily contains an independent set of size $3$, as does any $m$-cycle for $m\geq 6$.  One can easily verify by exhaustion that the $7$ forests pictured above are exactly those forests on $4$ or fewer vertices with independence number less than $3$.  This completes the proof.\hfill \qed}
\endgroup

We conclude the section with a classification theorem, combining Corollary \ref{3Pcor} and Theorem \ref{PartialConverse} into the following practical form.

\begin{theoremN}\label{MainTheorem}
A graph $F$ is isomorphic to the prime graph of some solvable group if and only if its complement $\overline{F}$ is $3$-colorable and triangle-free.
\end{theoremN}

\begin{proof}
If $F = \PG$ is the prime graph of some solvable group, then by Corollary \ref{3Pcor}, $\FG$ does not contain a $3$-path.  Thus by the Gallai-Roy theorem\cite[Thm.~7.17]{Gallai}, $\overline{\PG}$ is $3$-colorable.  Also, $\overline{\PG}$ is triangle-free by Lucido's Three Primes Lemma.  The converse is given by Theorem \ref{PartialConverse}.
\end{proof}

\section{Minimal Prime Graphs.}

For the remainder of the paper, we present an extended application of Theorem \ref{MainTheorem}, which demonstrates how graph theoretic properties can influence the structure of solvable groups.   

We first introduce a graph theoretic property we call \emph{minimal}.  Minimality was first observed as a property of the $5$-cycle while studying the exceptions to Corollary \ref{Girf}.  Solvable groups whose prime graphs are isomorphic to $5$-cycles were discovered to have certain group theoretic properties; in particular, these groups have Fitting length $3$.  We prove this in Section $4$.  We anticipate that this bound generalizes as a result of minimality, and in fact we find in Section $5$ that all solvable groups with minimal prime graphs have Fitting length $3$ or $4$.  This result reminds us of Lucido's similar conclusion in \cite[Prop.~3]{LucidoTree} concerning solvable groups with prime graphs of diameter $3$.

In this section, we outline some foundational results about minimality, culminating with the observation that a group with a minimal prime graph adheres to a conjectured bound on the number of prime divisors in the orders of its elements.

\begin{definition} If $G$ is a finite solvable group and $\PG$ satisfies \begin{enumerate}
\item[(a)] $|V(\PG)|>1$,
\item[(b)] $\PG$ is connected,
\item[(c)] $\PG\setminus \{pq\}$ is not the prime graph of any solvable group for any $p,q\in \PG$,
\end{enumerate}
then we say that $\PG$ is \emph{minimal}.
\end{definition}

One can easily verify that the $5$-cycle is the smallest minimal prime graph.  We now show that any minimal prime graph may be used to construct a new minimal prime graph of greater order.  \begin{definition} A \emph{linked vertex duplication} of a vertex $v$ in a graph $\Gamma$ is the graph constructed by adding a new vertex $u$ and a new edge $uv \in \Gamma$ such that $N^1(u) \setminus \{v\} = N^1(v) \setminus \{ u \}$.
\end{definition}

\begin{propositionN}\label{vertexduplication}
The family of minimal prime graphs is closed under linked vertex duplication.
\end{propositionN}

\begin{proof}
Let $\Gamma'$ be a minimal prime graph, and take $\Gamma$ to be the linked duplication of a vertex $v \in \Gamma'$.  We denote the added vertex to $\Gamma$ by $u$.

Let $\sigma : \Gamma \rightarrow \Gamma$ be the transposition permuting $u$ and $v$ and fixing all other vertices.  Since $N^1 (u) \setminus \{ v \} = N^1 (v) \setminus \{ u \}$, it follows that $\Gamma$ is the prime graph of some solvable group.  Furthermore, we have that $\sigma \in \Aut\Gamma$, so $\Gamma - v$ is isomorphic to $\Gamma'$.  Thus, by minimality of $\Gamma'$, we know $(\Gamma - v) - ux$ is not the prime graph of a solvable group for any $x \in N^1 (u)\setminus\{v\}$.  Furthermore, there exists an $x \in \Gamma$ so that $vx,ux\notin \Gamma$ and thus $\Gamma-uv$ contains an independent set of size $3$.  Therefore, $\Gamma$ is minimal.
\end{proof}

By starting with the $5$-cycle and repeatedly taking linked vertex duplications, the reader may produce many examples of minimal prime graphs.  It is important to note, however, that not all minimal can be obtained this way. One example is the Gr\"{o}tzsch graph with precisely one edge removed.  (The reader should note that the Gr\"{o}tzsch graph is otherwise known as the Mycielski graph of order $4$, or the triangle-free graph with chromatic number $4$ with the smallest number of vertices\cite{Mycielski}.)

Intuitively, groups with minimal prime graphs contain as many fixed-point-free actions as possible in a solvable group that is neither Frobenius nor $2$-Frobenius, as their Frobenius digraphs are saturated with arrows.  This rigid group structure causes their prime graphs to be somewhat well behaved.   We see that if a group $G$ has a minimal prime graph, then for all subgroups $K \leqslant G$, we have $\PrG{K} = \PG$ if and only if $V (\PrG{K}) = V (\PG)$.  Similarly, for all normal subgroups $N\unlhd G$, we have $\PrG{G/N} = \PG$ if and only if $V (\PrG{G/N}) = V (\PG)$.
\pagebreak
\begin{lemmaN}\label{not2colorable}
Suppose $G$ is a solvable group.  If $\PG$ is minimal, then $\overline{\PG}$ is not $2$-colorable.
\end{lemmaN}

\begin{proof}
Suppose that $\overline\PG$ is bipartite.  Since minimal prime graphs are connected, there exists at least one non-edge between the color classes.  Removing this edge from $\PG$ yields a graph whose with a bipartite, triangle-free complement, which contradicts the minimality of $\PG$.
\end{proof}

For the remainder of the paper, we fix the following notation.  Partition the vertices of $\FG$ into three sets $\mathcal{O}$, $\mathcal{D}$, and $\mathcal{I}$, where vertices in $\mathcal{O}$ have zero in-degree in $\FG$ (i.e., $p \in \mathcal {O}$ if and only if $N_\uparrow^1 (p)$ is empty), vertices in $\mathcal{D}$ have non-zero in- and out-degrees (i.e., $p \in \mathcal {D}$ if and only if both $N_\uparrow^1 (p)$ and $N_\downarrow^1 (p)$ are nonempty), and vertices in $\mathcal{I}$ have zero out-degree (i.e. $p \in \mathcal {I}$ if and only if $N_\downarrow^1 (p)$ is empty).  The following lemma shows that these sets actually form a partition.

\begin{lemmaN}\label{no singletons}
Let $G$ be a solvable group.  If $\PG$ is a minimal graph, then $\overline{\PG}$ contains no singleton vertices.
\end{lemmaN}

\begin{proof}
Any singleton vertex in $\overline\PG$ may be connected to any other vertex without creating a triangle or increasing the chromatic number of $\overline\PG$, so a minimal prime graph contains no singleton vertices.
\end{proof}

It follows from Lemma \ref{no singletons} that $\mathcal{O}$, $\mathcal{D}$ and $\mathcal{I}$ are pairwise disjoint, and by Lemma \ref{not2colorable} each must be nonempty.  In particular, these sets provide a $3$-coloring of $\overline{\PG}$, and it is this $3$-coloring we mean when we refer to a $3$-coloring of $\overline{\PG}$.
\vspace{5pt}
\begin{figure}[H]
\begin{center}
\begin{tikzpicture}[%
    auto,
    block/.style={
      rectangle,
      draw=blue,
      thick,
      fill=blue!20,
      text width=5em,
      align=center,
      rounded corners,
      minimum height=2em
    },
    block1/.style={
      rectangle,
      draw=blue,
      thick,
      fill=blue!20,
      text width=5em,
      align=center,
      rounded corners,
      minimum height=2em
    },
    line/.style={
      draw,thick,
      -latex',
      shorten >=2pt
    },
    cloud/.style={
      draw=red,
      thick,
      ellipse,
      fill=red!20,
      minimum height=1em
    }
  ]
    \node (p_1) at ( -1.425, 0) {$p_1$}; 
    \node (p_2) at ( 0.75, -1.5) {$p_2$};
    \node (p_3) at ( 0.55,1.5) {$p_3$};
    \node (p_4) at ( -2.1,1.5) {$p_4$};
    \node (p_5) at ( -1.35,-1.5) {$p_5$};
    \node (p_6) at ( 2.125,-1.5) {$p_6$};
    \node (p_7) at ( 1.7,0) {$p_7$};
    \node (p_8) at ( -1.05,1.5) {$p_8$};
    \node (p_9) at ( -0.011,0) {$p_9$};
    \node (p_{10}) at ( 1.82,1.5) {$p_{10}$};
    \node (p_{11}) at ( -2.75,-1.5) {$p_{11}$};
    \node (PG) at (-8.75705, -2.70645) {$\overline{\PG}$};
    \node (FG) at (-0.158273, -2.70645) {$\FG$};
    
    \node (p1) at ( -8.528, -0.724 ) {$p_1$}; 
    \node (p2) at ( -6.95, -0.441) {$p_2$};
    \node (p3) at ( -9.7824, 1.0862) {$p_3$};
    \node (p4) at ( -8.528, 0.724) {$p_4$};
    \node (p5) at ( -9.7824, -1.0862) {$p_5$};
    \node (p6) at ( -10.028, 0 ) {$p_6$};
    \node (p7) at ( -7.7894, 1.4742 ) {$p_7$};
    \node (p8) at ( -10.65, -1.9662 ) {$p_8$};
    \node (p9) at ( -10.65, 1.9662 ) {$p_9$};
    \node (p10) at (-7.7894, -1.4742 ) {$p_{10}$};
    \node (p11) at (-6.95, 0.441) {$p_{11}$};
    \node (O) at (3.373,1.5) {$\mathcal{O}$};
    \node (D) at (3.373,0) {$\mathcal{D}$};
    \node (I) at (3.373,-1.5) {$\mathcal{I}$};
   
    \begin{scope}[every path/.style={->}]
       \draw (p_4) -- (p_{11});
       \draw (p_4) -- (p_5);
       \draw (p_4) -- (p_2);
       \draw (p_4) -- (p_9);
       \draw (p_8) -- (p_1);
       \draw (p_8) -- (p_5);
       \draw (p_8) -- (p_6);
       \draw (p_3) -- (p_1);
       \draw (p_3) -- (p_5);
       \draw (p_3) -- (p_9);
       \draw (p_3) -- (p_7); 
       \draw (p_{10}) -- (p_5);      
       \draw (p_{10}) -- (p_2);      
       \draw (p_{10}) -- (p_6);      
       \draw (p_{10}) -- (p_{11});    
       \end{scope}
       \begin{scope}[every path/.style={-}]
       \draw (p_1) -- (p_{11}); 
       \draw (p_1) -- (p_2);
       \draw (p_9) -- (p_6);
       \draw (p_7) -- (p_6);
       \draw (p_7) -- (p_2);
       \draw (p_7) -- (p_{11});      
       \draw (p4) -- (p11);
       \draw (p4) -- (p5);
       \draw (p4) -- (p2);
       \draw (p4) -- (p9);
       \draw (p8) -- (p1);
       \draw (p8) -- (p5);
       \draw (p8) -- (p6);
       \draw (p3) -- (p1);
       \draw (p3) -- (p5);
       \draw (p3) -- (p9);
       \draw (p3) -- (p7); 
       \draw (p10) -- (p5);      
       \draw (p10) -- (p2);      
       \draw (p10) -- (p6);      
       \draw (p10) -- (p11);      
       \draw (p1) -- (p11); 
       \draw (p1) -- (p2);
       \draw (p9) -- (p6);
       \draw (p7) -- (p6);
       \draw (p7) -- (p2);
       \draw (p7) -- (p11);
       \end{scope}      
       \draw[thick,dotted]     (-3.45,-1) rectangle (2.873,-1.885);   
       \draw[thick,dotted]     (-3.45,0.485) rectangle (2.873,-0.4175);
       \draw[thick,dotted]     (-3.45,1.95) rectangle (2.873,0.9775);  
\end{tikzpicture}
\caption{An example of $3$-coloring the complement of a minimal prime graph.}
\end{center}
\end{figure}
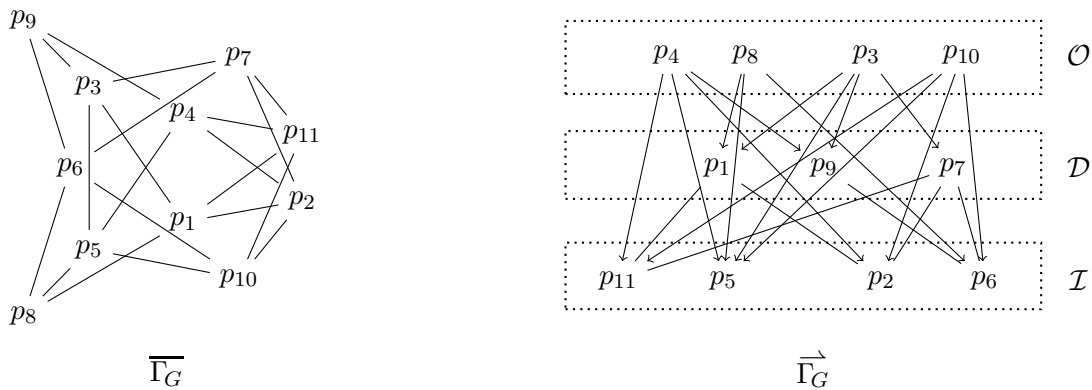

\vspace{-5pt}
We proceed with a technical lemma concerning the formation $q \leftarrow r \rightarrow p$ in $\FG$.

\begin{lemmaN} \label{Frob sub1}
Let $G$ be a solvable group and $p$, $q$, and $r$ be distinct primes dividing $|G|$.  Suppose that a Hall $\{p, r\}$-subgroup of $G$ is a Frobenius group whose Frobenius kernel is a Sylow $p$-subgroup of $G$.  Suppose additionally that $r \rightarrow q$ in $\FG$.  Then a Hall $\{p,q,r\}$-subgroup of $G$ is a Frobenius group whose Frobenius kernel is a Hall $\{p, q\}$-subgroup of $G$.  In particular, some Sylow $p$-subgroup and some Sylow $q$-subgroup of $G$ centralize each other.
\end{lemmaN}

\begin{proof}
Let $P$, $Q$, and $R$ be Sylow $p$-, $q$-, and $r$-subgroups, respectively, of $G$ so that $PQ$, $PR$, and $QR$ are subgroups.  We know that $PR$ is a Frobenius group with Frobenius kernel $P$.  Thus, $R$ is a Frobenius complement.  Also, $QR$ is either a Frobenius group or a $2$-Frobenius group of type $(r,q,r)$, but by Corollary \ref{type (p,q,p)}, it cannot be $2$-Frobenius.  Thus, $QR$ is a Frobenius group with Frobenius kernel $Q$.  It is not difficult to see that this implies that $PQR$ is a subgroup and in fact, it is a Frobenius group with Frobenius kernel $PQ$.  Since a Frobenius kernel is nilpotent, this implies that $P$ and $Q$ centralize each other.
\end{proof}

Next, we define the sets $\Pi = \{ p \in \FG : N_\uparrow^2 (p) \not= \emptyset\}$.  We must also introduce the binary octahedral group $2O := \langle r,s,t \mid r^2=s^3=t^4=rst \rangle$.  The group $2O$ is known under several guises.  It has order $48$, and it is the nonsplit extension of ${\rm SL}_2 (3)$ by a cyclic group of order $2$.  In this paper, we see $2O$ occur as a Frobenius complement.

\begin{propositionN}\label{HPiNilpotent}
Suppose $G$ is a solvable group.  If $\PG$ is minimal, then $\Pi \subseteq \mathcal {I}$ and $H_\Pi \leqslant \Fit{G}$.
\end{propositionN}

\begin{proof}
We begin by noting that $\Pi \subseteq \mathcal{I}$ by Corollary \ref{3Pcor}.  It suffices to prove that a Sylow $p$-subgroup $P$ is normal in $G$ for every $p \in \Pi$.  We do this by showing that the normalizer of $P$ contains a Sylow $s$-subgroup $S$ for every prime $s \in \pi (G)$.  If $s \rightarrow p$ in $\FG$, then $H_{sp}$ is Frobenius by Corollary \ref{adjacent}, so the normalizer of $P$ contains a Sylow $s$-subgroup of $G$.

Suppose now that $sp \in \PG$.  By minimality, the removal of $sp$ from $\PG$ must create a triangle in $\overline{\PG}$ or increase the chromatic number of $\overline{\PG}$.  Suppose $s$ and $p$ are in different color classes of $\overline{\PG}$.  Then removing $sp$ from $\PG$ does not increase the chromatic number of $\overline{\PG}$, so there exists a prime $t$ so that $st, tp \in \overline\PG$.   In the case that $s \in \mathcal{D}$, we have that $t \rightarrow s$ and $t \rightarrow p$ in $\FG$.  By Corollary \ref{adjacent}, we know that $H_{pt}$ is a Frobenius group.  Thus, we may apply Lemma \ref{Frob sub1} to see that some Sylow $s$-subgroup of $G$ normalizes $P$.  If $s \in \mathcal{O}$, then $s \rightarrow t \rightarrow p$ in $\FG$, so $H_{stp}$ is a $2$-Frobenius group of type $(p,t,s)$ by Lemma \ref{2path} and so, $P$ is normalized by a Sylow $s$-subgroup.

Let $s \in \mathcal{I}$.  Suppose there exists a prime $q \in N(p) \bigcap N(s)$.  We know that $H_{pq}$ is a Frobenius group by Corollary \ref{adjacent}, so we may apply Lemma \ref{Frob sub1} to show that some Sylow $s$-subgroup of normalizes $P$.  Assume now that $N(p) \bigcap N(s) = \emptyset$.  Then there exists a $2$-path $r \rightarrow q \rightarrow p$ in $\FG$ for which $qs \notin \overline \PG$.  Since $q \in \mathcal{D}$ and $s \in \mathcal{I}$, there exists a prime $t$ so that $tq, ts \in \overline\PG$ by minimality of $\PG$.  In particular, $t \in N_\uparrow^1(s)$.  By Corollary \ref{3Pcor}, $t \rightarrow q$ in $\FG$.  Let $H = H_{stqp}$.  Then $\FrG{H}$ consists of exactly the edges $t\rightarrow q$, $q\rightarrow p$, $t\rightarrow s$.  We conclude that $\PrG{H}$ has diameter $3$, so by \cite[Prop.~3]{LucidoTree}, either the Fitting length of $H$ is $3$ or the Fitting length of $H$ is $4$ and the binary octahedral group $2O$ is a normal section of $H$.

First, suppose that $\ell_F(H) = 4$ and $2O$ is a normal section of $H$.  Since $K = H_{tqp}$ is a $2$-Frobenius group of type $(p,q,t)$ by Theorem \ref{2path}, we know $q \ne 2$ by Lemma \ref{2 Frobenius}.  Let $N$ and $M$ be the normal subgroups of $H$ so that $M/N$ is isomorphic to $2O$.  It follows that $G/N$ has a central subgroup of order $2$.  First we observe that if $F_2 (K) \le N$, then we see that $K/N$ is a cyclic Sylow $t$-subgroup of $G/N$.  On the other hand, if $F_2 (K)$ is not contained in $K \cap N$, then $K/(K \cap N) \cong KN/N$ is either a Frobenius group or a $2$-Frobenius group.  Since both Frobenius groups and $2$-Frobenius groups have trivial center, we see that we $2$ cannot divide $|K|$.  We conclude that neither $t$ nor $p$ can be $2$, which forces $s = 2$.  Let $L$ be a Hall $\{2,t\}$-subgroup of $H$.  We know that $L$ is either a Frobenius group whose Frobenius kernel is a $2$-group or a $2$-Frobenius group of type $(t,2,t)$.  Since a Sylow $2$-subgroup is not cyclic it cannot be $2$-Frobenius of type $(t,2,t)$ by Lemma \ref{type (p,q,p)}.  On the other hand, if $L$ is a Frobenius group, then $LN/N$ cannot have a central subgroup of order $2$.  Therefore, we conclude that $s \ne 2$, and hence we cannot have that $\ell_F(H) = 4$.

It follows that $H$ has Fitting length $3$.  If $P$ is not normalized by $S$, then $P \not \leqslant \Fit{H}$.  So we have $t \rightarrow q \rightarrow p$ in the Frobenius digraph of $H/\Fit{H}$.  Thus $H/\Fit{H}$ contains a $2$-Frobenius group of type $(p,q,t)$, which necessarily has Fitting length $3$, a contradiction.  This final contradiction shows that $S$ normalizes $P$ for any $s\in \PG$, so $P$ is normal in $G$.
\end{proof}

We proceed to the first implication of the minimality criterion on the structure of a group, an application emphasizing the $3$-colorability condition of Theorem \ref{MainTheorem}.

\begin{definition}
Given a natural number $n$, denote by $\sigma(n)$ the number of prime divisors of $n$.  Define $\sigma(G) = \text{max} \{ \sigma (o(g)) : g \in G \}$.  Call a group \emph{$\sigma$-reduced} if every prime $p$ appears in at most one chief factor of $G$.
\end{definition}

\begin{theoremN}\label{3k}
For any solvable group $G$ for which $\PG$ is minimal, $\sigma(|G|) \leq 3 \sigma(G)$.
\end{theoremN}

\begin{proof}
It can be shown \cite[Lem.~16.17]{HuppertCharacters} that $G$ contains at least one $\sigma$-reduced subgroup $\Sigma$ so that $\sigma(|\Sigma|) = \sigma(|G|)$.  Since $\Sigma$ is a subgroup of $G$, we see that $pq \in \overline {\PrG{\Sigma}}$ for every $pq \in \overline {\PG}$.  Furthermore, since $\PG$ is minimal, every edge in $\PG$ must be present in $\PrG{\Sigma}$.  It follows that $\PrG{\Sigma}$ isomorphic to $\PG$, and since the orientation of edges in $\FG$ is closed under subgroups, we see in fact that $\FrG{\Sigma}$ is isomorphic to $\FG$.  Clearly $\sigma(\Sigma) \leq \sigma(G)$, so we may assume without loss of generality that $G$ is $\sigma$-reduced.

Because every prime appears in at most one chief factor of $G$, every Sylow subgroup of $G$ must be elementary abelian.  Every vertex in $\mathcal{O}$ and $\mathcal{D}$ serves as a Frobenius complement in some Hall subgroup of $G$ or an appropriate quotient group, so it follows that every Sylow $p$-subgroup for $p \in \mathcal{O} \bigcup \mathcal{D}$ is cyclic of prime order.  Then $H_{pq}$ is cyclic for any $p,q\in  \mathcal{O} \bigcup \mathcal{D}$ for which $pq \in \PG$. It follows that $H_\mathcal{O}$ and $H_\mathcal{D}$ are nilpotent.

To prove that $H_\mathcal{I}$ is nilpotent, by Proposition \ref{HPiNilpotent} it suffices to show that $H_{ps}$ is nilpotent for any $p\in \mathcal{I}$ and $s\in \mathcal{I}\setminus \Pi$.  We know that $N^1_\uparrow(p)\cap \mathcal{I}= \emptyset$ and $N^1_\uparrow(s)\cap \mathcal{D}= \emptyset$.  Thus if we add the edge $ps$ to $\overline\PG$, $\mathcal{O}, \mathcal{D}\cup \{s\}, \mathcal{I}\cup\{p\}$ remains an admissible $3$-coloring.  From minimality, $p$ and $s$ must then share an in-neighbor $q\in\mathcal{O}$. Thus $H_{ps}$ is either a Frobenius kernel or an upper kernel in a $H_{qps}$, so $H_{ps}$ is nilpotent.

We conclude that $n:=\text{max}\{|\mathcal{O}|,|\mathcal{D}|,|\mathcal{I}|\}\leq \sigma(G)$, and therefore \[\sigma(|G|)\leq 3n \leq 3\sigma(G).\qedhere \]
\end{proof}

\begin{remark}
Theorem \ref{3k} is motivated by the conjecture\cite[9]{Keller3} that, in fact, $\sigma(|G|)\leq 3\sigma(G)$ for every solvable group $G$.  One may notice that the above approach does not require minimality before proving that $\mathcal{I}$ is nilpotent.  Therefore this argument implies the general conjecture except for cases where $n< \sigma(G) < |\mathcal{I}|$, where $n=\text{max}\{|\mathcal{O}|,|\mathcal{D}|\}$.  In fact, $n$ can actually be strengthened to $n=\sum_k \text{max}\{|\mathcal{O}_k|,|\mathcal{D}_k|\}$, where $k$ runs over the components of $\FG [\mathcal{O} \bigcap \mathcal{D}]$, however in general the number of vertices in $\mathcal{I}$ can be much greater in non-minimal prime graphs.
\end{remark}

\section{$5$-cycles as Prime Graphs.}

We continue our discussion of minimal prime graphs by returning to solvable groups with $5$-cycle prime graphs.  Intuitively, we think of these groups as a primordial model for groups with minimal prime graphs, expecting many of the group theoretic properties resulting from minimality to stem from those exhibited here.  We develop this notion by showing that the $5$-cycle is not only the minimal prime graph of smallest order, but also that every minimal prime graph contains a $5$-cycle.  This shows that groups with $5$-cycle prime graphs occur as Hall subgroups in every solvable group with a minimal prime graph.  From this observation, we are able to quickly derive an upper bound on the Fitting length of any solvable group with a minimal prime graph, though this bound will be improved in the next section.  Most importantly, we prove that solvable groups with $5$-cycle prime graphs have Fitting length exactly $3$. 

\pagebreak
\begin{lemmaN}\label{inducedPentagon}
Every minimal graph contains an induced $5$-cycle.
\end{lemmaN}

\begingroup
{
\noindent\emph{Proof.}
Let $P,Q,$ and $R$ be the color classes of any $3$-coloring of $\overline{\PG}$.  Because $\overline{\PG}$ is not $2$-colorable, we can assume without loss of generality that for a prime $p \in P$, there exists some prime $q \in Q$ so that $pq\in \overline\PG.$

\setlength\intextsep{0pt}
\begin{wrapfigure}[9]{r}{0.4\textwidth}
\begin{center}
	\begin{tikzpicture}[scale=0.85]
    \node (pp) at (-0.587785, -0.809017) {$p^\prime$}; 
    \node (qp) at ( 0.587785, -0.809017) {$q^\prime$};
    \node (r) at ( 0,1) {$r$};
    \node (p) at (0.951057, 0.309017) {$p$};
    \node (q) at (-0.951057, 0.309017) {$q$};
    \node (pp2) at (3.41221, -0.809017) {$p^\prime$}; 
    \node (q2) at (  4.58779, -0.809017) {$q$};
    \node (rp) at (4,1) {$r^\prime$};
    \node (p2) at (3.04894, 0.309017) {$p$};
    \node (r2) at (4.95106, 0.309017) {$r$};
    \begin{scope}[every path/.style={-}]
       \draw (r) -- (p);
       \draw (r) -- (q);
       \draw (pp) -- (qp);
       \draw (r2) -- (q2);
       \draw (rp) -- (p2);
       \draw (pp2) -- (q2);
       \draw (pp) -- (qp);
    \end{scope} 
\end{tikzpicture}
\end{center}
\caption{Cases for Lemma \ref{inducedPentagon}.}
\end{wrapfigure}
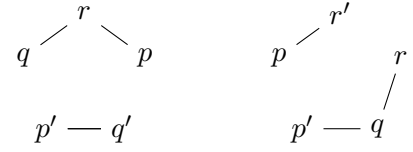
Suppose that there exists a prime $r\in R$ so that $rp, rq\in \PG$.  By minimality, removing the edge $rq$ will not yield the prime graph of a solvable group.  Since $r$ and $q$ are in different color classes in $\PG$, the graph $\overline{\PG}+rq$ admits the same $3$-coloring as $\overline \PG$.  Thus, there exists some prime $p'\in P$ so that $p'r, p'q\in \overline {\PG}$.  Similarly, since $rp \in \PG$, there exists some prime $q'\in Q$ so that $q'p, q'r \in \overline {\PG}$; whence by Lucido's Three Primes Lemma, $p'q'\in \PG$.  It follows that $\PG [\{p,p',q,q',r\}]$ is isomorphic to the $5$-cycle.

Now suppose that no prime $r \in R$ satisfies $pr, qr\in \PG$.  Again, since $\overline \PG$ is not $2$-colorable, we can assume without loss of generality that there exists a prime $r \in R$ so that $pr \in \overline \PG$, whence $rq \in \PG$.  Similarly, there exists a prime $r'\in R$ so that $qr' \in \overline\PG$, whence $r'p\in \PG$.  Thus, by minimality there exists a prime $p'\in P$ such that $p'r, p'r'\in \overline \PG$, whence $p'q\in \PG$.  It follows that $\PG[\{p,p',q,r,r'\}]$ is isomorphic to the $5$-cycle.\hfill \qed
}\endgroup

\begin{remark}
In a certain sense, the minimality criterion is a partial converse to Lucido's Three Primes Lemma.  Where this lemma asserts that $pq, pr\in \overline\PG$ implies that $qr \in \PG$, the minimality criterion asserts that when $q$ and $r$ are in different color partitions of $\overline\PG$, $qr \in \PG$ implies that there exists a $p$ so that $pq,pr\in \overline \PG$.  Viewed this way, Lucido's Three Primes Lemma together with the minimality criterion assures the existence of a self complementary object in every minimal graph.
\end{remark}

Next, we make a simple observation connecting Lemma \ref{inducedPentagon} to Proposition $3$ of \cite{LucidoTree}.  This allows us to easily derive an upper bound on the Fitting length of a group with a minimal prime graph.

\begin{corollaryN}\label{FittingLength}
Let $G$ be a group with a minimal prime graph.  Then $3\leq \ell_F(G)\leq 5$.
\end{corollaryN}

\begin{proof}
By Proposition \ref{HPiNilpotent}, we have that $H_\Pi\leqslant \Fit{G}$.  By Lemma \ref{inducedPentagon}, $\overline\PG$ contains at least one induced $5$-cycle, each of which has contains at least one vertex in $\Pi$.  Let $\Delta \subseteq \Pi$ be the smallest set of vertices such that $\PG[\Delta]$ contains no induced $5$-cycles.  Then there exists a pentagon that loses exactly one vertex in $\PrG{G/H_\Delta}$.  Hence $\PrG{G/H_\Delta}$ has diameter $3$.  By \cite[Prop.~3]{LucidoTree}, $\ell_F(G/H_\Delta)\leq 4$, so $\ell_F(G)\leq 5$.  Every $2$-Frobenius group has Fitting length $3$, so since every minimal prime graph contains a $2$-path, we obtain the lower bound $3\leq \ell_F(G)$.  Thus $3\leq \ell_F(G)\leq 5$.
\end{proof}

To show that solvable groups whose prime graph is isomorphic to a $5$-cycle have Fitting length $3$, we must first prove a general lemma.

\begin{lemmaN}\label{theOs}
Let $G$ be a solvable group such that $\PG$ is minimal.  Assume that $H = H_{srq}$ is a Hall $\{ q, r, s \}$-subgroup of
$G$ and that $H_{sr}$ and $H_{sq}$ are Hall $\{ r, s \}$- and $\{ q, s \}$ subgroups of $H$, respectively.
\begin{enumerate}
\item[(a)] If $s \rightarrow r \rightarrow q$ in $\FG$, then $\mathbf{O}_s (H_{srq}) = \mathbf{O}_s (H_{sr})$.
\item[(b)] If $q \leftarrow s \rightarrow r$ in $\FG$, then $\mathbf{O}_s (H_{sqr}) = \mathbf{O}_s (H_{sq}) = \mathbf{O}_s (H_{sr})$.
\item[(c)] If $s, q \in \mathcal{O}$ and $r \in N_\downarrow^1(s) \cap N_\downarrow^1(q)$.  Then $\mathbf{O}_s (H_{sqr}) = \mathbf{O}_s (H_{sr})$.
\end{enumerate}
\end{lemmaN}

\begin{proof}
Write $F = \Fit{H}$ and $K = H_{sr}$.  Note that in each case, $\Gamma_H$ is disconnected, so $H$ must be Frobenius or $2$-Frobenius.

\begin{enumerate}
\item[(a)]  By Lemma \ref{2path}, we have that $H$ is $2$-Frobenius of type $(q,r,s)$ and thus, $F = Q \times \mathbf{O}_s (H)$ where $Q$ is the Sylow $q$-subgroup of $H$.  Because $H/F$ is a Frobenius group with an $r$-group kernel, we see that $\mathbf{O}_s (H/F) = 1$.  Observe that $H = KQ$, so $K/(F\cap K) \cong H/F$.  We see that $K \cap F$ is a Hall $\{ r,s \}$-subgroup of $F$, whence $K \cap F = \mathbf{O}_s (H)$.  Since $\mathbf{O}_s (H)$ is a normal $s$-subgroup of $K$, we then have $\mathbf{O}_s (H)\leqslant \mathbf{O}_s (K)$.  Since $\mathbf{O}_s \left(K/\mathbf{O}_s (H)\right) = 1$, we conclude that $\mathbf{O}_s (H) = \mathbf{O}_s (K)$.

\item[(b)]  If $H$ is a Frobenius group, then $\mathbf{O}_s (H) = 1$ and $S$ is a Frobenius complement.  By Lemma \ref{type (p,q,p)}, this implies that $K$ and $H_{qs}$ cannot be $2$-Frobenius.  Thus, $K$ and $H_{qs}$ are Frobenius groups, and $\mathbf{O}_s (K) = \mathbf{O}_s (H_{qs}) = 1 = \mathbf{O}_s (H)$ as desired.  Thus, we may assume that $H$ is $2$-Frobenius.  This implies that $\mathbf{O}_s (H) > 1$ and $F = \mathbf{O}_s(H)$; so $F\leqslant \mathbf{O}_s (H_{qs})$ and $F\leqslant \mathbf{O}_s (K)$.  Observe that $K/F$ and $H_{qs}/F$ will be Frobenius groups, so $\mathbf{O}_s (K/F) = 1 = \mathbf{O}_s (H_{qs}/F)$.  This implies that $\mathbf{O}_s (K) \le F$ and $\mathbf{O}_s (H_{qs}) \le F$, and we obtain the desired equalities.

\item[(c)]  If $\mathbf{O}_s (H) = 1$, then $S$ is a Frobenius complement.  By Lemma \ref{type (p,q,p)}, this implies that $K$ cannot be $2$-Frobenius.  Thus, $K$ is a Frobenius group and $\mathbf{O}_s (K) = 1 = \mathbf{O}_s (H)$ as desired.  Thus, we may assume that $\mathbf{O}_s (H) > 1$.  This implies that $H$ is a $2$-Frobenius group.  In particular, $F = \mathbf{O}_s (H) \times \mathbf{O}_q (H)$.  We deduce that $K \cap F = \mathbf{O}_s (H)$.  We know that $K/(K \cap F) \cong KF/F$ is a Frobenius group, so $\mathbf{O}_s (K/(K \cap F)) = 1$.  This implies that $\mathbf{O}_s (K) \le K \cap F = \mathbf{O}_s (H) \le \mathbf{O}_s (K)$, and we have the desired equality.
\end{enumerate}
\end{proof}

\begin{propositionN}\label{5cycle}
If $G$ is a solvable group with $\PG$ isomorphic to the $5$-cycle, then the Fitting length of $G$ is $\ell_F(G)=3$.
\end{propositionN}
\begingroup{
\noindent\emph{Proof}. Let $p_1,\ldots,p_5$ be the prime divisors of $|G|$.  Because $\overline\PG$ is also isomorphic to the $5$-cycle, we have by Corollary \ref{3Pcor} that up to isomorphism there exists only one possible Frobenius digraph of $G$.  Without loss of generality, we label $\FG$ as shown in Figure \ref{5cycleDigraph}.

\setlength\intextsep{0pt}
\begin{wrapfigure}[11]{l}{0.35\textwidth}
\begin{center}
\begin{tikzpicture}[%
    auto,
    block/.style={
      rectangle,
      draw=blue,
      thick,
      fill=blue!20,
      text width=5em,
      align=center,
      rounded corners,
      minimum height=2em
    },
    block1/.style={
      rectangle,
      draw=blue,
      thick,
      fill=blue!20,
      text width=5em,
      align=center,
      rounded corners,
      minimum height=2em
    },
    line/.style={
      draw,thick,
      -latex',
      shorten >=2pt
    },
    cloud/.style={
      draw=red,
      thick,
      ellipse,
      fill=red!20,
      minimum height=1em
    }
  ]
    \node (p1) at ( 0, 0) {$p_1$}; 
    \node (p2) at ( 1.375, 0) {$p_2$};
    \node (p3) at ( 0,-1.375) {$p_3$};
    \node (p4) at ( 0,-2.75) {$p_4$};
    \node (p5) at ( 1.375,-2.75) {$p_5$};
    \node (O) at (-1.875,0) {$\mathcal{O}$};
    \node (D) at (-1.875,-1.375) {$\mathcal{D}$};
    \node (I) at (-1.875,-2.75) {$\mathcal{I}$};
    \node (I) at (2.875,-2.75) {$\text{ }$};
    
    \begin{scope}[every path/.style={->}]
       \draw (p1) -- (p3);
       \draw (p3) -- (p4); 
       \draw (p1) -- (p5);
       \draw (p2) -- (p4);
       \draw (p2) -- (p5);
       \draw[thick,dotted]     ($(p1.north west)+(-0.5,0.25)$) rectangle ($(p2.south east)+(0.5,-0.15)$);
	   \draw[thick,dotted]     ($(p3.north west)+(-0.5,0.25)$) rectangle ($(p3.south east)+(1.875,-0.15)$);     
       \draw[thick,dotted]     ($(p4.north west)+(-0.5,0.25)$) rectangle ($(p5.south east)+(0.5,-0.15)$);
    \end{scope}  
\end{tikzpicture}
\caption{Frobenius Digraph of $G$.}
\label{5cycleDigraph}
\end{center}
\end{wrapfigure}
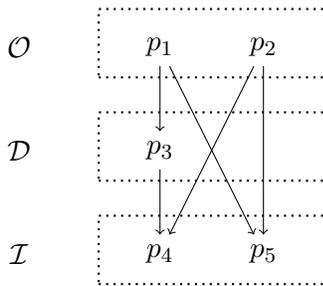
Observe that $\Pi = \{ p_4 \}$, $\II = \{ p_4, p_5 \}$, $\DD = \{ p_3 \}$, and $\OO = \{ p_1, p_2 \}$.   In particular, $H_{p_1p_3p_4}$ is $2$-Frobenius by Lemma \ref{2path}, so the Fitting length of $G$ is at least $3$.  We apply Proposition \ref{HPiNilpotent} to see that $P_4$ is normal in $G$.  Notice that the primes dividing $|\Fit{G}|$ must be adjacent to $p_4$ and to each other.

We first suppose that $H_{p_1p_3}$ is $2$-Frobenius, so $J = \mathbf{O}_{p_1} (H_{p_1p_3})$ is nontrivial.  We show that $J$ is normal in $G$ by showing that the normalizer of $J$ contains a Sylow subgroup for every prime dividing $|G|$.  It is obvious that $J$ is normalized by a Sylow $p_1$- and $p_3$-subgroup of $G$.  Let $H_{p_1p_3p_4}$ be a Hall $\{p_1,p_3,p_4\}$-subgroup containing $H_{p_1p_3}$.  By Lemma \ref{theOs}(a), we see that $J = \mathbf{O}_{p_1} (H_{p_1p_3p_4})$ and so $J$ is normalized by $P_4$.   Let $H_{p_1p_3p_5}$ be a Hall $\{p_1,p_3,p_5\}$-subgroup containing $H_{p_1p_3}$ and $H_{p_1p_5}$.   Applying Lemma \ref{theOs}(b), we see that $J = \mathbf{O}_{p_1} (H_{p_1p_3p_5}) = \mathbf{O}_{p_5} (H_{p_1p_5})$, and so $J$ is normalized by a Sylow $p_5$-subgroup of $G$.  Finally, we use Lemma \ref{theOs}(c) to see that $J = \mathbf{O}_{p_1} (H_{p_1p_2p_5})$, where $H_{p_1p_2p_5}$ is a Hall $\{ p_1, p_2, p_5 \}$-subgroup containing $H_{p_1p_5}$.  We conclude that a Sylow $p_2$-subgroup of $G$ normalizes $J$.  Thus we have shown that $J$ is normal in $G$.  It is not difficult to see that $J = \mathbf{O}_{p_1} (G)$.  Any prime divisor of $|\Fit{G}|$ must therefore be adjacent to $p_1$ in $\PG$, from which we conclude that $F = \Fit{G} = J \times P_4$.

We now argue that all the Sylow subgroups of $G/F$ are cyclic.  Observe that $F = \Fit{H_{p_1p_3p_4}}$, and since $H_{p_1p_3p_4}$ is $2$-Frobenius, we know by Lemma \ref{2 Frobenius} that all the Sylow subgroups of $H_{p_1p_3p_4}/F$ are cyclic.  This implies that the Sylow $p_1$- and $p_3$-subgroups of $G/F$ are cyclic.  Since $J = \mathbf{O}_{p_1} (H_{p_1p_5}) > 1$, we know that $H_{p_1p_5}$ is $2$-Frobenius, and thus a Sylow $p_5$-subgroup of $G$ is cyclic by Lemma \ref{type (p,q,p)}.  By Corollary \ref{adjacent}, we see that $H_{p_2p_5}$ must be a Frobenius group.  The Sylow $p_5$-subgroup of $H_{p_2p_5}$ is cyclic, so a Sylow $p_2$-subgroup must be cyclic as well.  Thus all the Sylow subgroups of $G/F$ are cyclic.  It follows that $G/F$ has Fitting length at most $2$, and so $G$ has Fitting length at most $3$.  For this case, this proves the theorem.

Next, suppose that $H_{p_1p_3}$ is a Frobenius group.  We show that $P_5$ is normal in $G$.  We know by Lemma \ref{type (p,q,r)} that $P_3$ is cyclic, so we have that $P_1$ is cyclic.  We can apply Lemma \ref{Frob sub1} to see that $H_{p_1p_3p_5}$ is a Frobenius group whose Frobenius kernel is a Hall $\{ p_3, p_5 \}$-subgroup of $G$.  This implies that the normalizer of $P_5$ contains Sylow $p_1$- and $p_3$-subgroups of $G$.  We see that $p_2 \in N_\uparrow^1 (p_4)$, so by Corollary \ref{adjacent}, a  $H_{p_2, p_4}$ is a Frobenius group, and by Lemma \ref{Frob sub1}, there is a Hall $\{p_2, p_4, p_5\}$-subgroup of $G$ that is Frobenius with kernel $P_4 P_5$.  This implies that $P_4$ centralizes $P_5$, and the normalizer of $P_5$ contains a Sylow $p_2$-subgroup of $G$.  We conclude that $P_5$ is normal in $G$.  Notice that all the prime divisors of $|F|$ are adjacent to $p_5$ in $\PG$, so $F = P_4 \times P_5$.

If all Sylow subgroups of $G/F$ are cyclic, then $G/F$ has Fitting length at most $2$, so $G$ has Fitting length at most $3$.  Thus we may assume that some Sylow subgroup of $G/F$ is not cyclic.  Since we know the Sylow $p_1$- and $p_3$-subgroups are cyclic, it must be that a Sylow $p_2$-subgroup is not cyclic.  We know that a Sylow $p_2$-subgroup $P_2$ is a Frobenius complement, and as it is not cyclic, we must have that $P_2$ is generalized quaternion.  Note that since $H_{p_1p_3}$ is a Frobenius group, if $p_3 = 3$, then we must have $p_1 = 2$, and this cannot occur since $p_2 = 2$.  Thus, we have that $p_3 \ne 3$.

Let $F_2/F=\Fit{G/F}$.  We know that $C_{G/F} (F_2/F) \le F_2/F$, so $G/F_2 \le \operatorname{Aut}(F_2/F)$.  If we assume that $F_2/F$ is cyclic, then $\operatorname{Aut}(F_2/F)$ will be abelian, so $G/F_2$ will also be abelian.  In this case, $G$ will have Fitting length at most $3$.  Thus, for the final step, we may assume that $F_2/F$ is not cyclic and work to obtain a contradiction.

Let $Q/F$ be the Sylow $2$-subgroup of $F_2/F$, and let $D/F$ be the Hall $2$-complement of $F_2/F$.  We know that $F_2/F = Q/F \times D/F$ and that $D/F$ is cyclic.  Thus, $G/F_2 \le \operatorname{Aut}(Q/F) \times \operatorname{Aut}(D/F)$.  We know that $\operatorname{Aut}(D/F)$ is abelian.  On the other hand, $Q/F$ is a subgroup of a generalized quaternion group, and therefore cyclic or generalized quaternion.  Thus $\operatorname{Aut}(Q/F)$ is abelian unless $Q/F$ is isomorphic to the quaternion group $Q_8$ of order $8$.  In this case, $\operatorname{Aut}(Q/F) \cong S_3$.  We have seen that the conclusion of the theorem holds if $G/F_2$ is abelian, so we may assume that $Q/F$ is isomorphic to $Q_8$.  If $C/F = C_{G/F} (Q/F)$, then $G/C \cong S_3$.  In particular, $G'C/C$ has order $3$, so $G'$ has a nontrivial Sylow $3$-subgroup.  Since $p_3 \ne 3$ and the only primes dividing $|G:F|$ are $p_1$, $p_2 = 2$, and $p_3$, this leaves us with $p_1 = 3$.

Notice $p_3$ does not divide $|G:C|$.  Hence, $C$ contains a Sylow $p_3$-subgroup $P_3$ of $G$.  Observe that $H_{p_1p_3} \cap F = 1$, so $H_{p_1p_3} \cong H_{p_1p_3}F/F$.  This implies that $p_1 = 3$ does not divide $|F_2/F|$, and thus, $D/F$ is a $p_3$-group.  Notice that $P_3$ centralizes $Q/F$ since $P_3 \le C$, and furthermore, since $D \le P_3F$ and $P_3 \cong P_3F/F$ is cyclic, we have since $P_3$ centralizes $D/F$.  Thus $P_3$ centralizes $F_2/F$, so $P_3F = D$.  Let $B/F = C_{G/F} (D/F)$.  Since $D/F$ is cyclic, we know that $G/B$ is abelian.  On the other hand, since $H_{p_1p_3}$ is a Frobenius group whose Frobenius kernel is $P_3$, we see that a Sylow $p_1$-subgroup of $B/F$ is trivial, and so $B$ has a trivial Sylow $p_1$-subgroup.  Since $G/B$ is abelian, $G' \le B$, and we conclude that $G'$ has a trivial Sylow $p_1$-subgroup.  However, recalling that $p_1=3$, we saw earlier that $G'$ has a nontrivial Sylow $p_1$-subgroup.  This contradiction completes the proof.
\hfill\qed}\endgroup

\section{Fitting Lengths and Minimal Prime Graphs.}

In this section, we expand on the ideas used in the proof of Proposition \ref{5cycle} to show that all solvable groups with minimal prime graphs have Fitting length at most $4$.  The lower bound of $3$ in Corollary \ref{FittingLength} is certainly best possible, since by Lemma \ref{not2colorable} there exists a $2$-Frobenius Hall subgroup of type $(p,q,r)$ that has Fitting length $3$.  The upper bound, on the other hand, may be improved by further examining the graph theoretic properties implied by minimality.  In our final theorem, we improve the upper bound from $5$ to $4$.

We must first present several results that give more detail regarding the structure of a minimal prime graph.  We begin by partitioning the sets $\mathcal{O}$ and $\mathcal{I}$ even further. As we will show in the following lemma, any $q \in \mathcal{O}$ is in contained in either $N_\uparrow^1 (p)$ or $N_\uparrow^2 (p)$ for any $p \in \mathcal{I}$.  With this in mind, we define $\mathcal{O}_1(p)=N_\uparrow^1 (p) \cap \mathcal{O}$, observing that $\mathcal{O} = \mathcal{O}_1(p)\cup N_\uparrow^2(p)$ is a disjoint union.  We then define the following sets. \begin{center}
$\displaystyle\mathcal{O}_1=\bigcup_{p\in \Pi}\mathcal{O}_1(p)$ \hspace{28pt} $\displaystyle\mathcal{O}_1^\star = \bigcap_{p \in \Pi}\mathcal{O}_1(p)$ \hspace{28pt} $\displaystyle\mathcal{O}_2=\bigcap_{p\in \Pi}N_\uparrow^2(p)$ \hspace{28pt} $\displaystyle\mathcal{O}_2^\star=\bigcup_{p\in \Pi}N_\uparrow^2(p)$ \end{center}
We observe that $\mathcal{O} = \mathcal{O}_1 \cup \mathcal{O}_2 = \mathcal{O}_1^\star \cup \mathcal{O}_2^\star$ are disjoint unions.  Finally, set $\Phi = \{p \in \mathcal{I} \mid N_\uparrow^2 (p) = \emptyset \}$.  We obtain as a result of part (b) in the following lemma that $\mathcal{I}=\Pi\cup \Phi$ is a disjoint union as well.

\begin{lemmaN}\label{LewisLemma456}
Let $G$ be a solvable group with a minimal prime graph.
\begin{enumerate}
\item[(a)] $\mathcal{O} \subseteq N_\uparrow^1 (p) \cup N_\uparrow^2 (p)$ for any prime $p \in \mathcal{I}$.
\item[(b)] $\mathcal{O} \subseteq N_\uparrow^1(p)$ for any prime $p \in \Phi$.
\item[(c)] If $s \in \mathcal{O}_2$, then $\mathcal{D} \subseteq N_\downarrow^1 (s)$.
\end{enumerate}
\end{lemmaN}

\begin{proof}${}$
Suppose that $s \in \mathcal{O} \setminus N_\uparrow^1 (p)$.  Then $\overline{\PG} + \{sp\}$ admits the same $3$-coloring as $\overline \PG$, so by minimality $\overline{\PG} + \{sp\}$ contains a triangle.  It follows that there exists a prime  $d \in \mathcal{D}$ so that $s \rightarrow d \rightarrow p$ in$ \FG$.  This proves part (a).  For part (b), we know that if $p \in \Phi$, then $N_\uparrow^2 (p) = \emptyset$, so part (b) is an immediate consequence of part (a).

To prove part (c), suppose that $sd \in \PG$ for some prime $d\in \mathcal{D}$.  Again by minimality, there exists a prime $p \in \mathcal{I}$ such that $s \rightarrow p \leftarrow d$.  Since $d \in N_\uparrow^1 (p)$, we see that $N_\uparrow^2 (p)$ is nonempty, and thus, $p\in \Pi$.  However, by definition of $\mathcal{O}_2$, we obtain $s\in N_\uparrow^2 (p)$, contradicting Lucido's Three Primes Lemma.  Thus, $s \rightarrow d$ in $\FG$.
\end{proof}

We now present an application of Lemmas \ref{2path} and \ref{type (p,q,r)} and Corollary \ref{adjacent} from which we obtain a characterization of the Sylow subgroups for primes in $\DD$ and in $\OO$.

\begin{lemmaN} \label{in DD}
Let $G$ be a solvable group such that $\Gamma_G$ is a minimal prime graph.
\begin{enumerate}
\item[(a)]  If $q \in \DD$ and $p \in N_{\downarrow}^1 (q)$, then $q \ne 2$, a Sylow $q$-subgroup of $G$ is cyclic, and a Hall $\{p,q\}$-subgroup of $G$ is a Frobenius group for which  Sylow $p$-subgroup is the Frobenius kernel and a Sylow $q$-subgroup is a Frobenius complement.
\item[(b)] If $s \in \OO_1$ and $t \in N_{\downarrow}^1 (s)$, then a Hall $\{s,t\}$-subgroup of $G$ is a Frobenius group whose Frobenius kernel is a Sylow $t$-subgroup of $G$ and a Sylow $s$-subgroup of $G$ is a Frobenius complement.  In particular, a Sylow $s$-subgroup of $G$ is either cyclic or generalized quaternion.
\item[(c)] If $s \in \OO_2^*$, then a Sylow $s$-subgroup of $G$ is not generalized quaternion.
\end{enumerate}
\end{lemmaN}

\begin{proof}
Suppose $q \in \DD$ and $p \in N_{\downarrow}^1 (q)$.  There exists $r \in \pi (G)$ so that $r \rightarrow q \rightarrow p \in \overrightarrow {\Gamma_G}$.  By Lemma \ref{2path}, a Hall $\{ p, q, r \}$-subgroup of $G$ is a $2$-Frobenius group of type $(r,q,p)$.  In light of the definition of type $(p,q,r)$ and Lemma \ref{type (p,q,r)}, we see that $q \ne 2$ and a Sylow $q$-subgroup of $G$ is cyclic, and a Hall $\{ p, q \}$-subgroup is a Frobenius group whose Frobenius kernel is a Sylow $p$-subgroup of $G$.  This proves part (a).

To prove part (b), we suppose $s \in \OO_1$ and $t \in N_{\downarrow}^1 (s)$.  Let $p$ be a prime in $\Pi$ so that $s \in \OO_1 (p)$.  There there exist primes $q, r \in \pi (G)$ so that $r \rightarrow q \rightarrow p \in \overrightarrow {\Gamma_G}$.  We now apply Corollary \ref{adjacent} to see that a Hall $\{ s, t \}$-subgroup of $G$ is a Frobenius group whose Frobenius kernel is a Sylow $t$-subgroup of $G$ and where a Sylow $s$-subgroup is a Frobenius complement.

Finally, suppose $s \in \OO_2^*$.  There exist primes $p,q \in \pi (G)$ so that $s \rightarrow q \rightarrow p \in \FG$.  This implies that a Hall $\{p, q, s \}$-subgroup is a $2$-Frobenius of type $(p,q,s)$.  By Lemma \ref{type (p,q,r)}, we conclude that a Sylow $s$-subgroup of $G$ is not generalized quaternion.
\end{proof}

With these results in mind, we are able to locate $2$ if $G$ has a Sylow $2$-subgroup that is generalized quaternion.

\begin{lemmaN}\label{LewisLemma12}
Suppose $G$ is a solvable group so that $\PG$ is a minimal prime graph.  If a Sylow $2$-subgroup of $G$ is generalized quaternion, then $2 \in \mathcal{O}_1^\star$.
\end{lemmaN}

\begin{proof}
Let $p$ be a prime in $\II$.  If $p \in \Pi$, then the assertion is true by Lemma \ref{type (p,q,r)}.  We need to deal with the case where $p$ is not in $\Pi$.  We know that if $q \in \OO$, then $q \rightarrow p$ by Lemma \ref{LewisLemma456}(b).  $H_{qp}$ is either Frobenius or $2$-Frobenius, and in either case, a Sylow $p$-subgroup is isomorphic to a Frobenius complement.  Sylow subgroups for primes in $\mathcal{D}$ are cyclic by Lemma \ref{in DD} (a).  Generalized quaternion groups are not cyclic and may not be Frobenius kernels, so $2\notin\mathcal{I}\cup \mathcal{D}$.  By Lemma \ref{in DD}(c), Sylow subgroups for primes in $\OO_2^*$ are not generalized quaternion. Therefore, since $\mathcal{O}=\mathcal{O}_1^\star\cup \mathcal{O}_2^\star$ is a disjoint union, we conclude that $2 \in \OO_1^*$.
\end{proof}

We can further characterize the Sylow subgroups in $\mathcal{O}_2$ by dividing them into those which are cyclic and those which are not.  Let $\mathcal{C}$ denote the primes in $s \in \mathcal{O}$ for which $S$ is cyclic and set $\mathcal{N} = \mathcal{O} \setminus \mathcal{C}$.  Note that when $s \in \mathcal{O}_2$ and $t \in N_\downarrow^1(s)$, we have by Lemma \ref{complement} that $H_{st}$ is Frobenius if and only if $s\in\mathcal{C}$.

\begin{theoremN}\label{Fit4}
Let $G$ be a solvable group such that $\PG$ is minimal.  Then $\ell_F(G) \leq 4$.  Furthermore, if $\ell_F(G) = 4$, then $2O$ is a normal section of $G$.
\end{theoremN}

\begin{proof}
We first show that it suffices to find a nilpotent normal subgroup $X$ such that all subgroups of $G/X$ are cyclic or generalized quaternion.  If such a subgroup $X$ exists, then let $E/X$ be the Fitting subgroup of $G/X$.  (This means that $E$ is the full pre-image in $G$ of the Fitting subgroup of $G/X$.)  It follows that the Sylow subgroups of $E/X$ are either cyclic or generalized quaternion.  Recall that the automorphism group of a cyclic group or a generalized quaternion group that is not the quaternions is abelian, and that the automorphism group of the quaternions is isomorphic to $S_3$, and so, has Fitting length $2$.  We know that $C_{G/X} (E/X) \leq E/X$, so $G/E$ is isomorphic to a subgroup of $\Aut{E/X}$.  It is not difficult to see that $\Aut{E/X}$ is isomorphic to a direct product of the automorphism groups of its Sylow subgroups, so $\Aut{E/X}$ has Fitting length at most $2$.  This implies that $G/E$ has Fitting length at most $2$.  Since $E/X$ and $X$ are both nilpotent, it follows that $\ell_F(G) \leq 4$.

Notice that if $G$ has Fitting length $4$, $S_3$ must be isomorphic to a subgroup of $G/E$.  We then have that the Sylow $2$-subgroup of $E/X$ is the quaternion group.  Let $D/X$ be the Hall $2$-complement of $E/X$.  We know that $E/D$ is isomorphic to the quaternion group $Q_8$ of order $8$.  Let $C/D = C_{G/D}\left(E/D\right)$.  Let $T/D$ be a Sylow $2$-subgroup of $CE/D$.  We know that $T/D$ is generalized quaternion since it is a nonabelian subgroup of a generalized quaternion group.  By Dedekind's lemma, $T = (C \cap T)E$, and this implies that $T = E$.  In particular, $(C \cap E)/D$ is a Sylow $2$-subgroup of $C/D$.  Let $B/D$ be a $2$-complement for $C/D$.  Since $(C \cap E)/D$ is normal and has order $2$, it is central.  This implies that $B/D$ is normal in $C/D$, and hence characteristic.  Hence $B$ is normal in $G$.  We have $B \cap E = D$, so $BE/B \cong E/D$ is isomorphic to $Q_8$.  Also, $G/CE = G/BE$ is congruent to $S_3$.  This implies that $G/B$ is an extension of $A_4$ by $Z_2$.  Since the Sylow $2$-subgroup is generalized quatenion, it cannot be a split extension. Thus, it is the nonsplit extension, and hence is isomorphic to $2O$.  It is not difficult to show that $G/X$ will have a normal subgroup isomorphic to $2O$.

By definition, every Sylow subgroup of $G$ for primes in $\CC$ are cyclic.  Applying Lemma \ref{in DD} (a) and (b), we know that the Sylow subgroups of $G$ for primes in $\DD \cup \OO_1$ are cyclic or generalized quaternion.

We first prove the theorem under the additional hypothesis that $\mathcal{N} = \emptyset$.  In this case, every prime divisor of $|G:H_\mathcal{I}|$ lies in $\mathcal{D} \cup \mathcal{C} \cup \mathcal{O}_1$.  If $H_{\II}$ is normal in $G$, then every Sylow subgroup of $G/H_{\II}$ will be cyclic or generalized quaternion.  We see that it suffices to show that $H_\mathcal{I}$ is nilpotent and normal in $G$, since then we can apply the first paragraph with $X = H_{\II}$.

If $\Phi = \emptyset$, then $\II = \Pi$.  Hence $H_{\II} = H_{\Pi}$ is normal and nilpotent by Proposition \ref{HPiNilpotent}, so we are done.  Therefore we suppose that $\Phi$ is nonempty.  It suffices to show that $P$ is normal in $G$ for every prime $p \in \Phi$.  We do this by showing that the normalizer of $P$ contains a Sylow subgroup for every prime in $\pi (G)$.  Suppose $s\in \mathcal{O}$, so $s \in \mathcal{C}$ or $s \in \mathcal{O}_1$.  By Lemma \ref{LewisLemma456}(b), $s \rightarrow p$ in $\FG$.  Thus, if $H_{sp}$ is a Hall $\{ s, p \}$-subgroup of $G$ containing $P$, then $H_{sp}$ is either Frobenius or $2$-Frobenius of type $(s,p,s)$.  If $H_{sp}$ is a $2$-Frobenius group, then Lemma \ref{type (p,q,p)} would imply that a Sylow $s$-subgroup is not a Frobenius complement;  however, this contradicts the observation we made in the second paragraph that Sylow subgroups for primes in $\CC \cup \OO_1$ are cyclic or generalized quaternion.  Thus, $P$ is normalized by some Sylow $s$-subgroup of $G$.

Next, consider a prime $q \in \mathcal{D}$, and by definition.  Take $r \in \OO_2 = \CC$ so that $r \rightarrow q$.  Again, by Lemma \ref{LewisLemma456} (b), we have that $r \rightarrow p$ in $\FG$.  We know that $H_{rq}$ is either a Frobenius group or a $2$-Frobenius group of type $(r,q,r)$, and since $R$ is cyclic, $H_{rq}$ cannot be a $2$-Frobenius group by Lemma \ref{type (p,q,p)}.  Thus, we may apply Lemma \ref{Frob sub1} to see that $H_{rqp}$ is a Frobenius group whose Frobenius kernel is a Hall $\{ q, p \}$-subgroup.  In particular, $P$ is centralized by some Sylow $q$-subgroup of $G$.

Suppose now that we have a prime $t \in \mathcal{I}$ such that $\mathcal{O}_1 (t)$ is nonempty.  Let $r \in \mathcal{O}_1 (t)$.  By Lemma \ref{LewisLemma456}(b) we see that $r \rightarrow  p$.  Applying Lemma \ref{in DD} (b), we see that $H_{rp}$ is a Frobenius group.  Using Lemma \ref{Frob sub1}, we see that $H_{rtp}$ is a Frobenius group whose Frobenius kernel is a Hall $\{ t, p \}$-subgroup.  In particular, $P$ is centralized by some Sylow $t$-subgroup of $G$.

Let $\Pi^\prime = \{q \in \II : \mathcal{O}_1 (q) = \emptyset \}$.  If $\Pi^\prime = \emptyset$, then $P$ is centralized by a Sylow $t$-subgroup for every prime $t \in \Pi$.  This implies that $P$ is normalized by a Sylow subgroup for every prime in $\pi (G)$, so $P$ is normal in $G$.  Since this holds for every prime in $\Phi$, we conclude that $H_{\II}$ is normal and nilpotent, and as we have observed, the result is proved in this case.

We now suppose that $\Pi^\prime$ is nonempty.  It follows that $\mathcal{O}_1^\star = \emptyset$.  In light of Lemma \ref{LewisLemma12}, a Sylow $2$-subgroup of $G$ cannot be generalized quaternion, so a Sylow $q$-subgroup is cyclic for every prime $q \in \mathcal{O}_1$.  Note that all prime divisors of $|G:H_\Pi|$ lie in $\mathcal{C} \cup \mathcal{O}_1 \cup \mathcal{D} \cup \Phi$.  We continue to consider a prime $p \in \Phi$.  We have seen that the Sylow subgroups for primes in $\CC \cup \OO_1 \cup \DD$ normalize $P$.  If $q \in \Phi$, then we know by Lemma \ref{LewisLemma456}(b) that $\OO_1 (q)$ contains $\OO$, so $\OO_1 (q)$ is not empty.  We see from the previous paragraph that a Sylow $q$-subgroup of $G$ normalizes $P$.  Hence $PH_\Pi/H_\Pi$ is normal in $G/H_\Pi$.  As this is true for all the primes in $\Phi$, it follows that $H_\mathcal{I}/H_\Pi$ is normal in $G/H_\Pi$ and nilpotent.

The prime divisors of $|G:H_{\II}|$ must lie in $\mathcal{C} \cup \mathcal{O}_1 \cup \mathcal{D}$, and we have seen that in this case the Sylow subgroups for primes in these three sets are cyclic.  Thus all Sylow subgroups of $G/H_{\II}$ are cyclic.  We then have that $G/H_{\II}$ has Fitting length at most $2$.  We claim that $D/H_{\II} = \Fit{G/H_{\II}}$ is a Hall $\DD$-subgroup of $G/H_{\Pi}$.  Consider a prime $s \in \CC \cup \OO_1$.  Since $\OO_1^* = \emptyset$, there is a prime $r \in \DD$, so that $s \rightarrow r$ in $\FG$.  Similarly, if $r \in \DD$, we can find a prime $s \in \CC \cup \OO_1$ so that $s \rightarrow r$ in $\FG$.  In both cases, we have that $H_{rs}$ is a Frobenius group.  This implies that $s$ does not divide $|D:H_{\Pi}|$ and $r$ does not divide $|G:D|$, which proves the claim.
  
We have seen that, when $p\in \Phi$, $P$ is centralized by a Sylow $d$-subgroup for every $d\in\DD$.  This implies that there is a Hall $\DD$-subgroup $H_\DD$ of $G$ that centralizes $P$.  Since this is true for every prime in $\Phi$, we can find a Hall $\Phi$-subgroup $H_\Phi$ of $G$ that is centralized by $H_{\DD}$.  We conclude that $D/H_{\Pi} \cong H_\Phi \times H_\DD$.  We have seen that $H_\Phi \cong H_\II/H_\Pi$ and $H_\DD \cong D/H_\II$ are nilpotent.  It follows that $D/H_{\Pi}$ is nilpotent.  Since $G/D$ and $H_{\Pi}$ are both nilpotent, we conclude that $\ell_F(G) \leq 3$.  This proves the result when $\mathcal{N}$ is empty.

We now suppose that $\mathcal{N}$ is nonempty.  Let $L = H_\Pi \times \prod_{s \in \mathcal{N}} \mathbf{O}_s (G)$.  Clearly, $L$ is nilpotent and normal in $G$.  We now show that all the Sylow subgroups of $G/L$ are cyclic or generalized quaternion. We then obtain the result by applying the first paragraph with $X = L$.  Observe that the primes dividing $|G:L|$ lie in $\NN \cup \CC \cup \OO_1 \cup \DD \cup \Phi$.  We have seen that Sylow subgroups for primes in $\CC \cup \OO_1 \cup \DD$ must be cyclic or generalized quaternion.  Thus, we need only consider primes in $\Phi \cup \NN$.

Consider first a prime $q \in \Phi$.  By Lemma \ref{LewisLemma456}(b), we know that $q \in N_\downarrow^1(s)$ for every prime $s \in \NN$.  Fix a prime $s \in \NN$.  Observe that $s \in \OO_2^*$, so a Sylow $s$-subgroup $S$ is not generalized quaternion.  Since $S$ is also not cyclic, we conclude that $S$ is not a Frobenius complement.  We know that $H_{qs}$ is either Frobenius or $2$-Frobenius of type $(s,q,s)$.  Since $S$ is not a Frobenius complement, we see that $H_{qs}$ is not a Frobenius group.  Thus, $H_{qs}$ must be $2$-Frobenius of type $(s,q,s)$, and we may use Lemma \ref{type (p,q,p)} to see that a Sylow $q$-subgroup is cyclic.  Hence, Sylow subgroups for every prime in $\Phi$ are cyclic.

We now consider a prime $s \in \NN$.  Since $s \in \mathcal{O}_2$, we know that there is a prime $r\in \mathcal{D}$ so that $s\rightarrow r$.  Consider a Hall $\{ r, s \}$-subgroup $H_{rs}$ of $G$.  Let $S$ be a Sylow $s$-subgroup contained in $H_{rs}$, and let $J = \mathbf{O}_s (H_{rs})$.  By Lemma \ref{LewisLemma12}, $S$ is not generalized quaternion.  Since we also know that $S$ is not cyclic, and $H_{rs}$ is either Frobenius or $2$-Frobenius, we deduce that $H_{rs}$ is not Frobenius.   Thus, $H_{rs}$ is a $2$-Frobenius group of type $(s,r,s)$.  By Lemma \ref{2 Frobenius}, $S/J$ must be cyclic.
If $J = {\bf O}_s (G)$, then $S \cap L = J$, and so, $S/J \cong SL/L$.  Since $SL/L$ is a Sylow $s$-subgroup of $G/L$, this implies that Sylow $s$-subgroups of $G/L$ are cyclic.  This implies that all of the Sylow subgroups of $G/L$ are cyclic or generalized quaternion.  Therefore we may once again obtain the result by applying the first paragraph with $X = L$.  Thus, the theorem will be proved once we prove that $J = {\bf O}_s (G)$.

We now prove that $J = \mathbf{O}_s (G)$.  Observe first that $\mathbf{O}_s (G) \leqslant J$.  We show that $J$ is normal in $G$ by showing that the normalizer of $J$ contains a Sylow subgroup for every prime in $\pi (G)$.  Once we know $J$ is normal in $G$, then we will have $J = {\bf O}_s (G)$.  By Lemma \ref{LewisLemma456} (b) and (c), if we have a prime  $q \in \mathcal{D} \cup \Phi$, then $q\in N_\downarrow^1(s)$.  Applying Lemma \ref{theOs}(b), we see that $J = \mathbf{O}_s (H_{sq})$, where $H_{sq}$ is a Hall $\{ s, q \}$-subgroup containing $J$.
In particular, some Sylow $q$-subgroup normalizes $\mathbf{O}_s(G)$.

If we have a prime $p \in \Pi$, then since $s \in \mathcal{O}_2$, we know that $p \in N_\downarrow^2(s)$.  Thus there exists a prime $q\in \mathcal{D}$ so that $s \rightarrow q \rightarrow p$ in $\FG$.   $K$ be a Hall $\{ p, q, s\}$-subgroup that contains $H_{sq}$.  Lemma \ref{theOs}(a) shows us that $J = \mathbf{O}_s(K)$.  We have that $J$ is normalized by the Sylow $p$-subgroup $P$ of $G$, which we recall is normal in $G$, and is thus centralized by $J$.  Since $K$ is a $2$-Frobenius group, we see that $C_K (P) = P \times J$.  In particular, we deduce that $J$ is a Sylow $s$-subgroup of $C_K (P)$.  Because this is true for every prime $p \in \Pi$, we conclude that $J$ is a Sylow $s$-subgroup of $C_G (H_\Pi)$.

If $t \in \mathcal{O}$ satisfies $N_\downarrow^1(s) \cap N_\downarrow^1(t) \ne \emptyset$, then we may apply Lemma \ref{theOs}(c) to see that some Sylow $t$-subgroup normalizes $J$.  By Lemma \ref{LewisLemma456}(c), this occurs for all the primes in $\mathcal{O}_2^\star$.  If $\Phi \ne \emptyset$, then this holds for all primes $t\in \mathcal{O}$ by Lemma \ref{LewisLemma456}(b).  Notice that this implies that the normalizer of $J$ contains Sylow subgroups for every prime dividing $|G|$, and so we have that $J$ is normal in $G$.  Thus, we may assume that $\Phi$ is empty.

We have shown that $J \leq C_G(H_\Pi)$.  If $J$ is normal in $C_G(H_\Pi)$, then $J$ is characteristic in $C_G(H_\Pi)$ since $J$ is a Sylow $s$-subgroup of $C_G (H_\Pi)$.  Because $H_\Pi$ is normal in $G$, it follows that $C_G(H_\Pi)$ is normal in $G$, and hence, $J$ is normal in $G$.  Thus it remains to be proven that $J \lhd C_G(H_\Pi)$.  We now show that every prime dividing $|C_G(H_\Pi)|$ is contained in $\Pi \cup \mathcal{O}_2$.  If $t \in \mathcal{O}_1 \cup \mathcal{D}$, then there is a prime $p\in \Pi$ so that $p\in N_\downarrow^1(t)$.  By Lemma \ref{2path} and Corollary \ref{3Pcor}, we know that $H_{pt}$ is a Frobenius group, so $t$ does not divide $|C_G (P)|$.  Since $C_G(H_\Pi) \leqslant C_G (P)$, it follows that $t$ does not divide $|C_G (H_\Pi)|$.  Since $\Phi$ is empty, we can conclude that the only primes that divide $|C_G (H_\Pi)|$ lie in $\Pi\cup \mathcal{O}_2$.  We have seen that the normalizer of $J$ contains a Sylow subgroup of $G$ for each prime in $\Pi \cup \mathcal{O}_2$, and thus contains a Sylow subgroup of $C_G(H_\Pi)$ for every prime divisor of $|C_G(H_\Pi)|$.  Thus $C_G(H_\Pi)$ normalizes $J$.  This proves the theorem.
\end{proof}

We now demonstrate that the upper bound in Theorem \ref{Fit4} is in fact best possible by explicitly constructing a group of Fitting length $4$ with a minimal prime graph.  We will see that the Frobenius digraph of the resulting group may be obtained by linked vertex duplication of the $5$-cycle.  Let $K = C_{11} \rtimes C_5$ be a Frobenius group and take $L = K \times 2O$, where $2O$ is the binary octahedral group.  Let $V_1$ be an absolutely irreducible $\mathbb{F}_{23} [L]$-module so that the fixed point space of the restriction of the module action to $C_{11} \times 2O$ is trivial.  The smallest such $V_1$ has dimension $10$.  Let $V_2$ be an absolutely irreducible $\mathbb{F}_{31}[L]$-module so that $C_5 \times 2O$ acts fixed point freely on $V_2$.  The smallest such $V_2$ has dimension $2$.  We have computationally verified that $\left( V_1 \times V_2 \right) \rtimes L$ has Fitting length $4$ and Frobenius digraph 
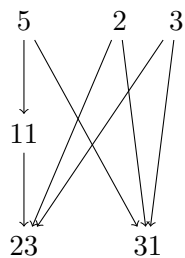
\begin{center}
\begin{tikzpicture}
    \node (p1) at ( 1.275, 1.5) {$2$}; 
    \node (p2) at ( 2.025, 1.5) {$3$};
    \node (p3) at ( 0,1.5) {$5$};
    \node (p4) at ( 0,0) {$11$};
    \node (p5) at ( 0,-1.5) {$23$};
    \node (p6) at ( 1.65,-1.5) {$31$};
    \begin{scope}[every path/.style={->}]
       \draw (p1) -- (p5);
       \draw (p2) -- (p5); 
       \draw (p1) -- (p6);
       \draw (p2) -- (p6);
       \draw (p3) -- (p4);
       \draw (p3) -- (p6);
       \draw (p4) -- (p5);       
    \end{scope}  
\end{tikzpicture}
\end{center}
which is minimal.

\section{Acknowledgements}

Much of this research was conducted under NSF-REU grant DMS $1005206$ by the first, second, fourth, and fifth authors during the Summer of 2012 under the supervision of the second author.  These authors graciously acknowledge the financial support of NSF, as well as the hospitality of Texas State University.  In particular, Dr. Jian Shen, the director of the REU program, is recognized for conducting an inspired and successful research program.  Also, these authors would like to thank Daniel Lenders for helpful discussion of the topic.  The third author learned about this work when visiting the second author for one week in February/March $2013$ and then was able to improve the upper bound on the Fitting length in Section $2$ from $5$ to $4$.  He would like to thank Texas State University for its hospitality and support.

\end{document}